%% file: main.tex
\documentclass{article}
\usepackage{settings}

\title{AUTOEQUIVALENCES OF DERIVED CATEGORIES OF BIELLIPTIC SURFACES}
\author{Yuki Tochitani}
\date{}
\begin{document}

\maketitle

\makeatletter
\def\@makefnmark{}
\makeatother


\begin{abstract}
    We determine the generators of the autoequivalence group of the derived category of coherent sheaves on a bielliptic surface over an algebraically closed field of arbitrary characteristic. As a consequence, we prove that any algebraic variety derived equivalent to such a surface is isomorphic to the surface itself.
\end{abstract}

\input{sections/section1/section1_main}
\input{sections/section2/section2_main}
\input{sections/section3/section3_main}
\input{sections/section4/section4_main}
\bibliography{references}
\bibliographystyle{amsalpha}

\bigskip

email: y.tochitani0112@gmail.com

\end{document}

%% file: sections/section1/section1_main.tex
\section{Introduction}
\input{sections/section1/intro}
\input{sections/section1/notation}

%% file: sections/section1/intro.tex
\subsection{Motivations and results}

Let $X$ be a smooth projective variety over an algebraically closed field $k$ of arbitrary characteristic, and denote by $D(X)$ the bounded derived category of coherent sheaves on $X$.  
We study the group of autoequivalences of the $k$-linear triangulated category $D(X)$, denoted by $\Auteq D(X)$. 

This group always contains the subgroup $\Pic (X) \rtimes \Aut (X) \times \mathbb{Z}$, which is generated by pullbacks along automorphisms of $X$, tensoring with line bundles and the shift functor. 
The autoequivalences belonging to this subgroup are said to be \textit{standard}.
For example, if the (anti-)canonical divisor of $X$ is ample, then every autoequivalence of $D(X)$ is a standard autoequivalence \cite[Theorem 3.1]{MR1818984}. 
Otherwise, $\Auteq D(X)$ does not necessarily contain non-standard autoequivalences.

The group $\Auteq D(X)$ has also been studied for abelian varieties \cite{MR1921811}, K3 surfaces of Picard rank 1  \cite{MR3592689}, certain types of elliptic surfaces \cite{MR3568337}, \cite{MR3652778} and surfaces of general type \cite{MR2198807}. However, many cases remain to be explored.

In this article, we study $\Auteq D(X)$ for a bielliptic surface $X$.  A bielliptic surface is a smooth projective surface with $K_{X} \equiv 0, b_{2}=2$ and the smooth Albanese map. It is well known that any bielliptic surface has exactly two elliptic fibrations.
The group $\Auteq D(X)$ has non-standard autoequivalences, which we call \textit{relative Fourier--Mukai transforms} along the elliptic fibrations.  
Bielliptic surfaces are classified into two types: \textit{cyclic} and \textit{non-cyclic}, and the generators of $\Auteq D(X)$ are described for cyclic case in $\Char(k) =0$ (\cite[Theorem 1.2]{potter2017derived}).
We generalize this result to any type in arbitrary characteristic.

\begin{mthm} \label{MT1}
    Let $X$ be a bielliptic surface over an algebraically closed field $k$ of arbitrary characteristic. Then $\Auteq D(X)$ is generated by standard autoequivalences and relative Fourier--Mukai transforms along the two elliptic fibrations.
\end{mthm}

The main tool in \cite[Theorem 1.2]{potter2017derived} is the canonical cover of $X$. In the case $\Char(k) =0$, the canonical cover of a bielliptic surface is an abelian surface, and Potter uses some geometric and homological properties of abelian surfaces. 
However, in the case $\Char(k) >0$, the canonical cover of a bielliptic surface is either an abelian surface or a bielliptic surface. 
Thus we need to take an alternative approach to generalize \cite[Theorem 1.2]{potter2017derived}.

As a consequence of \cref{MT1}, we obtain the following theorem.
Let $\Phi \in \Auteq D(X)$ be an autoequivalence. Then $\Phi$ induces an automorphism $\Phi_{M}$ on the numerical Chow group $\mathbb{Z} \oplus \Num (X) \oplus \mathbb{Z} (\simeq \mathbb{Z}^{4})$. 
Let $\Gamma (\lambda_{i})$ denote the subgroup of $\SL_{2}(\mathbb{Z})$ as
$\Gamma (\lambda_{i}) \coloneqq \inset{
    \begin{bmatrix}
        c & a \\
        d & b \\
    \end{bmatrix} 
    \in \SL_{2}(\mathbb{Z})}{
    a \in \lambda_{i} \mathbb{Z}
}$
, where $\lambda_{i}$ is given in \cref{bie}.

\begin{thm}
    [=\cref{ESC}]
    Let $X$ be a bielliptic surface over $k$ of arbitrary characteristic. Then we have a short exact sequence
    \[
        1 
        \to  \Aut (X) \times \mathbb{Z} [2]
        \to \Auteq D(X) 
        \xrightarrow{\pi} \{ A_{1} \otimes A_{2} \in \GL_{4}(\mathbb{Z}) |
        A_{i} \in \Gamma (\lambda_{i}) \}
        \to 1,
    \]
    where $\pi$ is defined by $\pi (\Phi) = \Phi_{M}$
    for $\Phi \in \Auteq D(X)$.
\end{thm}

Another interesting problem related to $D(X)$ is what the Fourier--Mukai partners of $X$ are. 
A variety $Y$ is called \textit{a Fourier--Mukai partner} of $X$ if there exists a $k$-linear triangulated equivalence $D(X) \simeq D(Y)$. 

For example, if the (anti-)canonical divisor of $X$ is ample, $X$ has only the trivial Fourier--Mukai partner (i.e. $X$ itself) \cite[Theorem 2.5]{MR1818984}. In general, however, there exist pairs of varieties that are mutually Fourier--Mukai partners but not isomorphic to each other. 
The first example of such pairs was discovered by Mukai in \cite{nmj/1118786312}.
He proved that an abelian variety $A$ and its dual $\hat{A}$ are derived equivalent. 

In the case $\Char(k)=0$, if a surface $X$  has a non-trivial Fourier--Mukai partner, then $X$ is either a K3 surface, an abelian surface, or a relatively minimal elliptic surface with non-zero Kodaira dimension \cite[Theorem 1.1]{MR1827500}, \cite[Theorem 1.6]{10.4310/jdg/1090351323}. In particular, a bielliptic surface in $\Char (k) =0$ has only the trivial Fourier--Mukai partner. 
In the case $\Char (k) \ge 5$, bielliptic surfaces are liftable, and it is followed by this property that there do not exist non-trivial Fourier--Mukai partners of a bielliptic surface in $\Char (k) \ge 5$ (\cite[Theorem 1.2]{MR4247995}).
In the case $\Char (k) =2$ or $3$, although the  liftability of bielliptic surfaces does not necessarily hold,
an argument in \cref{MT1} shows the next theorem.

\newcounter{breakpoint7}
\setcounter{breakpoint7}{\value{mthm}}

\begin{mthm} \label{MT2}
    A bielliptic surface $X$ over an algebraically closed field $k$ of arbitrary characteristic has no non-trivial Fourier--Mukai partners.
\end{mthm}

%% file: sections/section1/notation.tex
\subsection{Notation and conventions}

All varieties are defined over an algebraically closed field $k$.  
Unless otherwise stated, varieties are smooth and projective. A point on a variety always means a closed point. 
The group $\Num(X)$ is the group of divisors on a variety $X$ modulo numerical equivalence. 
For a coherent sheaf $\mathcal{E}$, the sheaf $\Tor (\mathcal{E})$ is the torsion part of $\mathcal{E}$.
We denote by $D(X)$ the bounded derived category of coherent sheaves on an algebraic variety $X$.  Equivalences between these categories always mean $k$-linear triangulated equivalences.
For matrices $X = (x_{ij}) \in M_{m,n}(\mathbb{Z}), Y = (y_{ij})\in M_{p,q}(\mathbb{Z})$, we define the tensor product as
\[
    X \otimes Y \coloneqq 
    \begin{bmatrix}
        x_{11}Y & x_{12}Y & \dotso & x_{1n}Y \\
        \vdots & \vdots & \ddots & \vdots \\
        x_{m1}Y & x_{m2}Y & \dotso & x_{mn}Y \\
    \end{bmatrix}.
\]
Note that the equality
\[
    (A \otimes C) (B \otimes D) = AB \otimes CD
\]
holds for matrices
$A \in M_{m,n}(\mathbb{Z}), 
B \in M_{n,l}(\mathbb{Z}),
C \in M_{p,q}(\mathbb{Z})$ and 
$D \in M_{q,r}(\mathbb{Z}).$

\paragraph{Acknowledgements.}
The author is grateful to Hokuto Uehara for valuable guidance and support.
Thanks are also due to Yuta Takashima for helpful comments on the manuscript.

%% file: sections/section2/section2_main.tex
\section{Preliminaries}

\input{sections/section2/bielliptic}

\input{sections/section2/Euler}
\input{sections/section2/matrix}
\input{sections/section2/canonical_cover}

\input{sections/section2/relative_FM}

%% file: sections/section2/bielliptic.tex
\subsection{Bielliptic surfaces} \label{bie}

\begin{definition}
    A minimal projective surface $X$ is called a bielliptic surface if $K_{X} \equiv 0, b_{2}=2$ and all fibers of the Albanese map are smooth elliptic curves.
\end{definition}
A classification of bielliptic surfaces is known (\cite[Theorem 4]{MR0491719}). 
A bielliptic surface $X$ is obtained as quotient of the product of two elliptic curves $A$ and $B$ by a finite subgroup $G$ of $A$, where $G$ acts on $A$ by translations and on $B$ by automorphisms.
Since $\Aut(B) = B \rtimes \Aut_{0}(B)$, we can write $G = H \times G_{0}$, where $G_{0}$ is a finite subgroup of $\Aut_{0}(B)$ and $H$ is a finite subgroup of $B$. 
The bielliptic surface $X = (A \times B) /G$ has two natural elliptic fibrations
\[
    f_{1} \colon ( A \times B ) / G \to A /G, 
    \hspace{10pt}
    f_{2} \colon ( A \times B ) / G \to B /G.
\]
Here, $f_{1}$ is the Albanese map and hence is smooth.
Let $F_{i} \in \Num (X)$ be the numerical class of a smooth fiber of $f_{i}$ for $i \in \{ 1,2\}$.

\begin{lem} \label{ROT}
    We have $F_{1} \cdot F_{2} = \rank (G)$.
    Here, $\rank (G)$ is defined as $\dim_{k}R$, where $G = \Spec R$.
\end{lem}

\begin{proof}
    Consider the following commutative diagram:
    \[
        \begin{tikzcd}
            A \arrow[d, "\pi_1"'] 
            & A \times B \arrow[l, "p_1"'] \arrow[r, "p_2"] \arrow[d, "\pi"]
            & B \arrow[d, "\pi_2"] \\[1.2em]
            A/G 
            & (A \times B)/G \arrow[l, "f_1"] \arrow[r, "f_2"']
            & B/G
        \end{tikzcd}
    \]

    Let $P_{1} \in A, P_{2} \in B$ be points and $Q_{1} \coloneqq \pi_{1}(P_{1}), Q_{2} \coloneqq \pi_{2}(P_{2})$.
    Then, 
    \begin{align*}
        \rank(G) F_{1} \cdot F_{2}
        &= \rank(G) f_{1}^{\ast}(Q_{1}) \cdot f_{2}^{\ast}(Q_{2}) \\
        &= \deg(\pi) f_{1}^{\ast}(Q_{1}) \cdot f_{2}^{\ast}(Q_{2}) \tag*{(\cite[Theorem 8.1]{Pink-FiniteGroupSchemes})} \\ 
        &= \pi^{\ast} f_{1}^{\ast}(Q_{1}) \cdot \pi^{\ast} f_{2}^{\ast}(Q_{2}) \tag*{(projection formula)} \\
        &= p_{1}^{\ast} \pi_{1}^{\ast}(Q_{1}) \cdot p_{2}^{\ast} \pi_{2}^{\ast}(Q_{2}) .
    \end{align*}
    For $i \in \{1,2\}$, let $V_{i} \coloneqq p_{i}^{\ast}(P_{i}) \in \Num (A \times B)$ denote the numerical class of the fiber of the point $P_{i} \in A$. 
    Then
    \[
        p_{1}^{\ast} \pi_{1}^{\ast}(Q_{1}) 
        \equiv 
        \deg (\pi) V_{1}
        =
        \rank (G) V_{1}.
    \]
    Similarly, we can write
    \[
        p_{2}^{\ast} \pi_{2}^{\ast}(Q_{2}) 
        \equiv \rank (G) V_{2}.
    \]
    Since $V_{1} \cdot V_{2} =1$, it follows that
    \begin{eqnarray*}
        \rank (G) F_{1} \cdot F_{2}
        &=&
        p_{1}^{\ast} \pi_{1}^{\ast}(Q_{1}) \cdot p_{2}^{\ast} \pi_{2}^{\ast}(Q_{2}) \\
        &=&
        \rank (G)^{2} V_{1} \cdot V_{2} \\
        &=& \rank (G)^{2}. 
    \end{eqnarray*}
    Therefore,  $F_{1} \cdot F_{2} = \rank (G)$.
\end{proof}

For $i \in \{ 1,2 \}$, put
\[
    \lambda_{i} \coloneqq 
    \min \{ F_{i} \cdot D \mid D \mbox{ is the numerical class of an effective divisor on } X \}.
\]
Since $f_{1}$ is the Albanese map, it is smooth.
Put 
\[
    \mu \coloneqq \lcm \{ m_{i} | m_{i} \mbox{ is the multiplicity of a multiple fiber of } f_{2} \}.
\] 
Note that
\begin{equation}\label{F2-mu}
    \frac{1}{\mu} F_{2} \in \Num(X).
\end{equation}

Then, by the proof of \cref{NUM}, we can calculate
\[
    \lambda_{1} = \frac{F_{1} \cdot F_{2}}{\mu}, 
    \hspace{5pt}
    \lambda_{2} = \mu.
\]
Combining \cref{ROT} and \cite{MR0491719}, 
we can classify bielliptic surfaces as in \cref{table}.
In \cref{table}, “multiplicities” denotes the multiplicities of all multiple fibers of $f_{2}$.
For instance, consider the case where $X$ is a bielliptic surface of type (2,1). Then we have $\mu =2$ and 
\[
    \lambda_{1} = \frac{F_{1} \cdot F_{2}}{\mu} = \frac{2}{2}=1, 
    \hspace{5pt}
    \lambda_{2} = \mu=2.
\]

\begin{table}[htbp] 
    \centering
    \small
    \setlength{\tabcolsep}{4pt}
    \renewcommand{\arraystretch}{0.9}
    \begin{tabular}{| c | c | c | c | c | c | c | c | c | c |} 
        \hline
        \textbf{Type} & \( G_0 \) & \( H \) & \(\operatorname{char}(k)\) & \(\operatorname{ord}(\omega_X)\) & \( \lambda_1 \) & \( \lambda_2 \)  &  \( F_{1} \cdot F_{2} \)  & \textbf{multiplicities} & \(\mu\) \\
        \hline
        (2,1) & \( \mathbb{Z}/2\mathbb{Z} \) & \( e \) & \( \neq 2 \) & 2 & 1 & 2 & 2 & \{2,2,2,2\} & 2 \\
        (2,1) & \( \mathbb{Z}/2\mathbb{Z} \) & \( e \) & \( = 2 \) & 1 & 1 & 2 & 2 & \{2\} or \{2,2\} & 2 \\
        (2,2) & \( \mathbb{Z}/2\mathbb{Z} \) & \( \mathbb{Z}/2\mathbb{Z} \) & \( \neq 2 \) & 2 & 2 & 2 & 4 & \{2,2,2,2\} & 2\\ 
        (2,$\mu_2$) & $\mathbb{Z}/2\mathbb{Z}$ & $\mu_2$ & $=2$ & 1 & 2 & 2  & 4 & \{2\} or \{2,2\} & 2 \\
        \hline
        (3,1) & \( \mathbb{Z}/3\mathbb{Z} \) & \( e \) & \( \neq 3 \) & 3 & 1 & 3 & 3  & \{3,3,3\} & 3 \\
        (3,1) & \( \mathbb{Z}/3\mathbb{Z} \) & \( e \) & \( = 3 \) & 1 & 1 & 3 & 3 & \{3\} & 3 \\
        (3,3) & \( \mathbb{Z}/3\mathbb{Z} \) & \( \mathbb{Z}/3\mathbb{Z} \) & \( \neq 3 \) & 3 & 3 & 3  & 9 & \{3,3,3\} & 3 \\ 
        \hline
        (4,1) & \( \mathbb{Z}/4\mathbb{Z} \) & \( e \) & \( \neq 2 \) & 4 & 1 & 4 & 4  & \{2,4,4\} & 4\\
        (4,1) & \( \mathbb{Z}/4\mathbb{Z} \) & \( e \) & \( =2 \) & 1 & 1 & 4 & 4  & \{4\} & 4\\
        (4,2) & \( \mathbb{Z}/4\mathbb{Z} \) & \( \mathbb{Z}/2\mathbb{Z} \) & \( \neq2 \) & 4 & 2 & 4  & 8 & \{2,4,4\} & 4\\ 
        \hline
        (6,1) & \( \mathbb{Z}/6\mathbb{Z} \) & \( e \) & \( \neq 2,3 \) & 6 & 1 & 6 & 6  & \{2,3,6\} & 6\\
        (6,1) & \( \mathbb{Z}/6\mathbb{Z} \) & \( e \) & \( = 2 \) & 3 & 1 & 6 & 6 &\{6,3\} & 6\\
        (6,1) & \( \mathbb{Z}/6\mathbb{Z} \) & \( e \) & \( = 3 \) & 2 & 1 & 6  & 6 & \{6,2\} & 6\\
        \hline
    \end{tabular}
    \caption{Types of bielliptic surfaces.} \label{table}
\end{table}

Following Potter’s terminology, a bielliptic surface of type $(s,t)$ is said to be \textit{cyclic} when $t=1$ and \textit{non-cyclic} when $t \neq1$.

\begin{rmk}
    \cref{NUM_lem} and \cref{NUM} have been proved over the complex number field (see \cite[Lemma 1.5, Theorem 1.4]{{MR1038716}}) and a similar argument applies over any algebraically closed field.
\end{rmk}

\begin{lem} \label{NUM_lem}
    Let $X$ be a bielliptic surface
    and $\alpha \in \mathbb{Z}_{\ge 1}$. If $(1/\alpha)F_{2} \in \Num (X)$, then we have $\mu \ge \alpha$.
\end{lem}

\begin{proof}
    It suffices to show that if 
    \[
        \frac{1}{F_{1} \cdot F_{2}}F_{2} \in \Num (X),
    \]
    then $F_{1} \cdot F_{2} = \mu$.
    Indeed, for any $\alpha \in \mathbb{Z}_{\ge 1}$ with $(1/\alpha)F_{2} \in \Num (X)$, we have $F_{1} \cdot F_{2}/ \alpha \in \mathbb{Z}$. Moreover, according to \cref{table}, if $\alpha > \mu$, then $\alpha = F_{1} \cdot F_{2}$, hence the statement follows.
    
    Assume that 
    \[
        \frac{1}{F_{1} \cdot F_{2}}F_{2} \in \Num (X).
    \]
    Take $D \in \Div (X)$ such that 
    \[
        D \equiv F_{1} + \frac{1}{F_{1} \cdot F_{2}} F_{2}. 
    \]

    \textit{Step 1.}
    We prove that $D$ is effective.  
    Since $\chi (\mathcal{O}(D)) = (1/2) D^{2} = 1 >0$, we have $h^{0} (\mathcal{O}(D)) >0$ or $h^{2} (\mathcal{O}(D)) >0$.
    We have $h^{2} (\mathcal{O}(D)) = h^{0} (\mathcal{O}(-D + K)) =0$,  
    hence $h^{0} (\mathcal{O}(D)) >0$.

    \textit{Step 2.}
    We prove that $D$ is reducible. 
    Assume that $D$ is irreducible.
    First, we note that $D$ is normal.
    Indeed, assume $D$ is not normal and let $p$ be a singular point of $D$. Put $F_{p} \coloneqq f_{1}^{-1}(f_{1}(p))$. Then we obtain
    \[
        1 = D \cdot F_{1} \ge (D \cdot F_{1})_{p} \ge \mu_{p}(D) \ge 2,
    \]
    which is a contradiction.
    We express $\varphi \coloneqq f_{1}|_{D} \colon D \to A / G$ as the composition of 
    purely inseparable morphism $\varphi_{\ins} \colon D \to D_{\sep}$ 
    and 
    separable morphism $\varphi_{\sep} \colon D_{\sep} \to A/G$.  
    As $D \cdot F_{1} =1$, we have $|\varphi_{\sep} ^{-1} (p) | =1$ for any $p \in A/G$, hence $\deg (\varphi_{\sep}) =1$. It follows that $\varphi_{\sep}$ is an isomorphism. 
    On the other hand, because 
    $\varphi_{\ins}$ is purely inseparable, 
    we have $g(D) = g(A/G) =1$. By the adjunction formula, we get
    \[
        0 = 2 g(D) -2 = D \cdot (D +K ) = D^{2} =2,
    \]
    which is a contradiction.

    \textit{Step 3.}
    We prove that there exists an effective divisor $D^{\prime} \in \Div (X)$ such that 
    \[
        D^{\prime} \equiv \frac{1}{F_{1} \cdot F_{2}} F_{2}. 
    \]
    As $D$ is effective, for $D = \sum_{i=1}^t m_{i}D_{i}$ where $m_{i} >0$ and each $D_{i}$ is a prime divisor, it suffices to show that there exists $i$ such that $D_{i} = F_{1}$, 
    because $D - D_{i}$ is an effective divisor.
    Assume $D_{i} \ne F_{1}$ for all $i$. As $D_{i} \cdot F_{1} > 0$ and $D \cdot F_{1} =1$, we obtain $t=1$. However, this contradicts the reducibility of $D$.

    \textit{Step 4.}
    We prove that $F_{1} \cdot F_{2} = \mu$. Define $D^{\prime}$ as above. Let $D^{\prime} = \sum_{i=1}^{t'} a_{i}D_{i}^{\prime}$ be defined to satisfy $a_{i} > 0$ and each $D_{i}^{\prime}$ is a prime divisor. It follows that
    \[
        0 
        = F_{2} \cdot \expar{ \frac{1}{F_{1} \cdot F_{2}} F_{2}}
        = F_{2} \cdot D^{\prime}
        = \sum_{i=1}^{t'} a_{i} ( F_{2} \cdot D_{i}^{\prime}).
    \]
    Therefore, for any $i$, we obtain $F_{2} \cdot D_{i}^{\prime} =0$ and there exists a point $Q_{i} \in B/G$ such that $D_{i}^{\prime} = f_{2}^{-1}(Q_{i})$, where $f_{2}^{-1}(Q_{i})$ is reduced. On the other hand, it holds that
    \[
        1 
        = F_{1} \cdot \expar{ \frac{1}{F_{1} \cdot F_{2}}F_{2}}
        = F_{1} \cdot D^{\prime}  
        = \sum_{i=1}^{t'} a_{i} (F_{1} \cdot f_{2}^{-1}(Q_{i})).
    \]
    Thus we obtain $t^{\prime}=1, a_{1}=1$ and $F_{1} \cdot f_{2}^{-1}(Q_{1}) = 1$. Then $f_{2}^{\ast}(Q_{1}) = (F_{1} \cdot F_{2}) f_{2}^{-1}(Q_{1})$, and hence $f_{2}$ has a multiple fiber of multiplicity $F_{1} \cdot F_{2}$. Consequently, by \cref{table}, we conclude that $F_{1} \cdot F_{2} = \mu$.
\end{proof}

\begin{thm} \label{NUM}
    Let $X$ be a bielliptic surface.
    Then 
    \[
        \Num (X) = \left( \frac{1}{\lambda_{1}}F_{1} \right) \mathbb{Z} \oplus 
        \left( \frac{1}{\lambda_{2}}F_{2} \right) \mathbb{Z}.
    \]
\end{thm}

\begin{proof}
    It is enough to prove that $\Num (X)$ is generated by $\frac{\mu}{F_{1} \cdot F_{2}} F_{1}$ and $\frac{1}{\mu} F_{2}$, because we can compute $\lambda_{1}, \lambda_{2}$ as follows:
    \[
        \lambda_{1} = \expar{\frac{1}{\mu}F_{2}} \cdot F_{1} = \frac{F_{1} \cdot F_{2}}{\mu},
    \]
    \[
        \lambda_{2} = \expar{\frac{\mu}{F_{1} \cdot F_{2}}F_{1}} \cdot F_{2} = \mu.
    \]
    
    Since $\dim (\Num (X) \otimes \mathbb{Q})=2$ by $b_{2}(X)=2$, for any $W \in \Num (X)$, there exist integers $n_{1}, n_{2}, s_{1}, s_{2}, \ \\ t_{1}, t_{2} \in \mathbb{Z}$ such that
    \[
        W \equiv \expar{n_{1} + \frac{t_{1}}{s_{1}} } F_{1} + \expar{n_{2} +  \frac{t_{2}}{s_{2}} } F_{2}
    \]
    with $t_{i}/s_{i}$ irreducible fractions. 
    Note that $\sum_{i=1}^2 (n_{i}+ (t_{i} / s_{i})) F_{i} \in \Num (X)$ is equivalent to $\sum_{i=1}^2 (t_{i} / s_{i}) F_{i} \in \Num (X)$.

    \textit{Case 1.}
    We prove the statement in the case of $F_{1} \cdot F_{2} = \mu$.
    Take any $D = \sum_{i=1}^2 (t_{i} / s_{i}) F_{i} \in \Num (X) \otimes \mathbb{Q}$ with $t_{i}/s_{i}$ irreducible fractions. Assume $D \in \Num (X)$.
    It follows that
    \[
        \mathbb{Z} \ni D \cdot F_{1} = \frac{t_{2}}{s_{2}} F_{2} \cdot F_{1} = \frac{t_{2} \mu}{s_{2}}.
    \]
    Thus $s_{2}$ divides $\mu$, and we can write 
    $D = (t_{1} / s_{1}) F_{1} + (t_{2}^{\prime} / \mu) F_{2}$ with some $t_{2}^{\prime} \in \mathbb{Z}$.
    It follows from \cref{F2-mu}, $D \in \Num (X)$ if and only if $(t_{1}/ s_{1}) F_{1} + (1/ \mu ) F_{2} \in \Num (X)$. 
    If this holds, it follows that
    \[
        \mathbb{Z} \ni
        \chi \expar{ \mathcal{O} \expar{ \frac{t_{1}}{s_{1}} F_{1} + \frac{1}{\mu} F_{2}}} 
        =
        \frac{1}{2} \expar{ \frac{t_{1}}{s_{1}} F_{1} + \frac{1}{\mu} F_{2}}^{2} \\
        =
        \frac{t_{1} F_{1} \cdot F_{2}}{s_{1} \mu}
        =
        \frac{t_{1}}{s_{1}}.
    \]
    Hence, $t_{1} / s_{1} \in \mathbb{Z}$. Then, $\frac{\mu}{F_{1} \cdot F_{2}} F_{1} = F_{1}$ and $(1/ \mu) F_{2}$ form a basis of $\Num (X)$.

    \textit{Case 2.}
    We prove the statement in the case of $F_{1} \cdot F_{2} \neq \mu$.
    Take any $D = \sum_{i=1}^2 (t_{i} / s_{i}) F_{i} \in \Num (X) \otimes \mathbb{Q}$ with $t_{i}/s_{i}$ irreducible fractions. Assume $D \in \Num (X)$.
    First, we claim $s_{2} \ne F_{1} \cdot F_{2}$.
    Assume, for contradiction, that $s_{2} = F_{1} \cdot F_{2}$.
    Combining \cref{F2-mu} with
    \[
        (F_{1} \cdot F_{2}, \mu ) = 
        (4,2), 
        (8,4) \mbox{ or }
        (9,3),
    \]
    we have $\frac{t_{1}}{s_{1}}F_{1} +  \frac{l}{F_{1} \cdot F_{2}}F_{2} \in \Num (X)$ with some $l \in \{ -1,1 \}$.
    Then
    \[
        \chi \expar{ \mathcal{O} \expar{ \frac{t_{1}}{s_{1}} F_{1} + \frac{l}{F_{1} \cdot F_{2}} F_{2}}} 
        = \frac{t_{1}l}{s_{1}} \in \mathbb{Z}.
    \]
    This gives 
    \[
        \frac{1}{F_{1} \cdot F_{2}} F_{2} \in \Num (X),
    \]
    which contradicts \cref{NUM_lem}.
     Next, we claim $\frac{\mu}{F_{1} \cdot F_{2}} F_{1}  \in \Num (X)$.
    It follows that
    \[
        D \cdot F_{1} = \frac{t_{2}F_{2} \cdot F_{1}}{s_{2}} \in \mathbb{Z}.
    \]
    Thus $s_{2}$ divides $F_{1} \cdot F_{2}$.
    By the previous claim, we can write
    \[
        D = \frac{t_{1}}{s_{1}} F_{1} + \frac{t_{2}^{\prime}}{\mu} F_{2}
    \]
    with some $t_{2}^{\prime} \in \mathbb{Z}$. Here, we have
    \[
        D \cdot F_{1} 
        = \frac{t_{2}^{\prime}}{\mu} F_{1} \cdot F_{2}
        \in \expar{ \frac{F_{1} \cdot F_{2}}{\mu} } \mathbb{Z}.
    \]
    Since the intersection pairing is non-degenerate, it follows that 
    \[
        \frac{\mu}{F_{1} \cdot F_{2}} F_{1} \in \Num (X).
    \]
    Note that $D \in \Num (X)$ is equivalent to $\frac{t_{1}}{s_{1}} F_{1} + \frac{1}{\mu} F_{2} \in \Num (X)$.
    Then 
    \[
        \mathbb{Z} \ni \chi \expar{ \mathcal{O} \expar{ \frac{t_{1}}{s_{1}} F_{1} + \frac{1}{\mu} F_{2}}} 
        =
        \frac{t_{1}F_{1} \cdot F_{2}}{s_{1}\mu}.
    \]
    Thus $s_{1}$ divides $\frac{F_{1} \cdot F_{2}}{\mu}$ and the assertion holds.
\end{proof}

\begin{rmk} \label{ED2}
    Let $X$ be a bielliptic surface and $D$ be an effective divisor on $X$. Then we have $D^{2} \ge 0$. Indeed, we can write 
    \[
        D \equiv \frac{s_{1}}{\lambda_{1}} F_{1} + \frac{s_{2}}{\lambda_{2}} F_{2}
    \]
    for some $s_{1}, s_{2} \in \mathbb{Z}$. We note $s_{1}, s_{2} \ge 0$ since $D \cdot F_{1} \ge 0$ and $D \cdot F_{2} \ge 0$. Then we obtain
    \[
        D^{2} = 2 s_{1} s_{2} \ge 0.
    \]
    It follows from this that $\Num (X)$ is an even lattice.
\end{rmk}

%% file: sections/section2/Euler.tex
\subsection{The Euler form on surfaces}

Let $X$ be a variety. For objects $\mathcal{E}, \mathcal{F} \in D(X)$, we define the Euler form as
\[
    \chi ( \mathcal{E}, \mathcal{F} ) 
    \coloneqq 
    \sum_{i} (-1)^{i} \dim \Hom_{D(X)}^{i}(\mathcal{E}, \mathcal{F}).
\]
For a object $\mathcal{E} \in D(X)$ we define its \emph{Mukai vector} as
\[
    v (\mathcal{E}) \coloneqq \ch (\mathcal{E}) \cdot \sqrt{\td(X)},
\]
where $\td (X) \in A(X) \otimes \mathbb{Q}$ is the Todd class of $X$.
We define on the rational Chow group $A(X) \otimes \mathbb{Q}$ a bilinear form $\langle \cdot , \cdot \rangle$, called \emph{Mukai pairing}, by setting
\[
    \langle v , w \rangle \coloneqq - \int_{X} v^{\ast} \cdot w \cdot \exp \expar{ \frac{1}{2} \ch_{1}(X)}
\]
for $v,w \in A(X) \otimes \mathbb{Q}$.
In the surface case, the Riemann--Roch theorem yields
\begin{eqnarray*}
    \chi ( \mathcal{E}, \mathcal{F} )  
    &=&
    \int_{X} \ch(\mathcal{E}^{\vee}) \cdot \ch(\mathcal{F}) \cdot \td (X) \\
    &=&
    \ch_{0}(\mathcal{E}) \ch_{2}(\mathcal{F})
    - \ch_{1}(\mathcal{E}) \ch_{1}(\mathcal{F})
    + \ch_{0}(\mathcal{F}) \ch_{2}(\mathcal{E}) \\
    &+& \frac{1}{2}(
    \ch_{0}(\mathcal{F}) \ch_{1}(\mathcal{E})
    - \ch_{0}(\mathcal{E}) \ch_{1}(\mathcal{F})
    ) \cdot K_{X}
    + \ch_{0}(\mathcal{E}) \ch_{0}(\mathcal{F}) \chi(\mathcal{O}_{X}) \\
    &=& - \langle v(\mathcal{E}) , v(\mathcal{F}) \rangle.
\end{eqnarray*}
In particular, if $X$ is a bielliptic surface, it follows that
\[
    \chi ( \mathcal{E}, \mathcal{F} ) = 
    \ch_{0}(\mathcal{E}) \ch_{2}(\mathcal{F})
    - \ch_{1}(\mathcal{E}) \ch_{1}(\mathcal{F})
    + \ch_{0}(\mathcal{F}) \ch_{2}(\mathcal{E})
\]
since $K_{X} \equiv 0$ and $\chi(\mathcal{O}_{X}) = 0$. 

%% file: sections/section2/matrix.tex
\subsection{Action of autoequivalences on $\mathbb{Z} \oplus \Num (X) \oplus \mathbb{Z}$} \label{Matrix}

Let $X$,$Y$ be varieties and 
\[
    p_{Y} \colon Y \times X \to Y,
    \quad
    p_{X} \colon Y \times X \to X
\]
be the canonical projections.

\begin{definition}
    Let $\mathcal{P} \in D(Y \times X)$. The induced \emph{integral functor} is the functor
    \[
        \Phi_{\mathcal{P}} \colon D(Y) \to D(X) ;  \mathcal{E} \mapsto \mathbb{R}p_{X \ast} (p_{Y}^{\ast} \mathcal{E} \overset{\mathbb{L}}{\otimes} \mathcal{P}).
    \]
    An integral functor $\Phi_{\mathcal{P}} \colon D(Y) \to D(X)$ is called a \emph{Fourier--Mukai functor} if it is an equivalence.
\end{definition}

Let
\[
    F \colon D(Y) \to D(X)
\]
be a fully faithful functor. By \cite[Theorem 3.2.1]{MR1998775}, it is known that 
there exists an object $\mathcal{P} \in D(Y \times X)$ unique up to isomorphism such that $F$ is isomorphic to $\Phi_{\mathcal{P}}$.

Let $\Phi_{\mathcal{P}} \colon D(Y) \to D(X)$ denote an integral functor. By the Grothendieck--Riemann--Roch theorem, $\Phi_{\mathcal{P}}$ induces the commutative diagram
\[
    \begin{CD}
            D(Y) @>{\Phi_{\mathcal{P}}}>> D(X) \\
            @V{v}VV    @V{v}VV \\
            A(Y) \otimes \mathbb{Q}   @>{\Phi_{v(\mathcal{P})}^{A}}>>  A(X) \otimes \mathbb{Q}
    \end{CD}
\]
and the automorphism
$\Phi_{v(\mathcal{P})}^{A}$ is defined by
\[
   \Phi_{v(\mathcal{P})}^{A} (\alpha) 
    = \pi_{X \ast} (\pi_{Y}^{\ast}\alpha \cdot v(\mathcal{P})).
\]

In the following, we assume that $X=Y$ and $X$ is a bielliptic surface. 
Since the Todd class of a bielliptic surface is trivial, we have $v(\mathcal{E}) = \ch(\mathcal{E})$
for $\mathcal{E} \in D(X)$. 
The Chow group decomposes by codimension as
\[
    A(X) \otimes \mathbb{Q} = ( A^{0}(X) \otimes \mathbb{Q}) \oplus ( A^{1}(X) \otimes \mathbb{Q}) \oplus ( A^{2}(X) \otimes \mathbb{Q}).
\]
Here, $A^{0}(X) \simeq \mathbb{Z}$ and $A^{1}(X) = \Pic (X)$.
By using the natural surjective degree map
$\deg \colon A^{2}(X) \to \mathbb{Z}$
and the canonical surjection
$\Pic (X) \to \Num (X)$,
we obtain a natural surjection
\[
    A^{0}(X) \oplus A^{1}(X) \oplus A^{2}(X) \to \mathbb{Z} \oplus \Num (X) \oplus \mathbb{Z}.
\]
Note that $\Num (X)$ is an even lattice by \cref{ED2}.
By abuse of notation, we use the same symbol
$\Phi_{\ch(\mathcal{P})}^{A}$
for the induced automorphism on $\mathbb{Z} \oplus \Num (X) \oplus \mathbb{Z}$.
Consequently, we obtain the following commutative diagram:
\[
    \begin{CD}
        D(X) @>{\Phi_{\mathcal{P}}}>> D(X) \\
        @V{\ch}VV    @V{\ch}VV \\
        \mathbb{Z} \oplus \Num(X) \oplus \mathbb{Z}   @>{\Phi_{\ch(\mathcal{P})}^{A}}>>  \mathbb{Z} \oplus \Num(X) \oplus \mathbb{Z}.
    \end{CD}
\]
In view of \cref{NUM}, we further regard it as an element of $\Aut (\mathbb{Z}^{4}) \simeq  \GL_{4}(\mathbb{Z})$.
Let $\Phi_{M} \in \GL_{4}(\mathbb{Z})$ denote the matrix associated to $\Phi \in \Auteq D(X)$. 
For autoequivalence $\Phi \in \Auteq D(X)$, we adopt the following notation for convenience:
\[
    \Phi_{M} (a,b,c,d) \coloneqq  \Phi_{M} 
    \begin{bmatrix}
        a \\
        b \\
        c \\
        d \\
    \end{bmatrix}.
 \]
 
Put
\[
    \bm{v_{1}} \coloneqq \Phi_{M}(1,0,0,0) , \quad
    \bm{v_{2}} \coloneqq \Phi_{M}(0,1,0,0),
\]
\[
    \bm{v_{3}} \coloneqq \Phi_{M}(0,0,1,0), \quad
    \bm{v_{4}} \coloneqq \Phi_{M}(0,0,0,1).
\]
Then we have
\[
    \Phi_{M} =
    \begin{bmatrix}
        \bm{v_{1}} & \bm{v_{2}} & \bm{v_{3}} &\bm{v_{4}}\\
    \end{bmatrix}.
\]
We can calculate Mukai pairings on $\mathbb{Z} \oplus\Num (X) \oplus \mathbb{Z} \simeq \mathbb{Z}^{4}$ as follows:
\begin{eqnarray*}
    \exgen{ \expar{r, 
    \frac{s_{1}}{\lambda_{1}}F_{1} + \frac{s_{2}}{\lambda_{2}}F_{2}, t} , 
    \expar{r^{\prime}, \frac{s_{1}^{\prime}}{\lambda_{1}} F_{1} + \frac{s_{2}^{\prime}} {\lambda_{2}}F_{2},t^{\prime}}}
    &=&
    s_{1}s_{2}^{\prime} + s_{2}s_{1}^{\prime} - rt^{\prime} - tr^{\prime} \\
    &\eqqcolon& 
    \big\langle (r,s_{1},s_{2},t) , (r^{\prime},s_{1}^{\prime},s_{2}^{\prime},t^{\prime}) \big\rangle_{\mathbb{Z}^{4}}. 
\end{eqnarray*}
Since autoequivalences preserve Mukai pairings by \cite[Proposition 5.44]{MR2244106}, it follows that
\begin{equation} \label{V14}
    \langle \bm{v_{1}} , \bm{v_{4}} \rangle_{\mathbb{Z}^{4}} 
    = \langle \bm{v_{4}} , \bm{v_{1}} \rangle_{\mathbb{Z}^{4}}
    = \langle (0,0,0,1) , (1,0,0,0) \rangle_{\mathbb{Z}^{4}}
        = -1,
\end{equation}
\begin{equation} \label{V23}
    \langle \bm{v_{2}} , \bm{v_{3}} \rangle_{\mathbb{Z}^{4}} 
    = \langle \bm{v_{3}} , \bm{v_{2}} \rangle_{\mathbb{Z}^{4}} 
    = \langle (0,0,1,0) , (0,1,0,0) \rangle_{\mathbb{Z}^{4}}
    = 1
\end{equation}
and for the other pairs of $i,j$,
\begin{equation} \label{Vij}
    \langle \bm{v_{i}} , \bm{v_{j}} \rangle_{\mathbb{Z}^{4}} 
    = \langle \bm{v_{j}} , \bm{v_{i}} \rangle_{\mathbb{Z}^{4}} 
    = 0.
\end{equation}

\begin{rmk} \label{LM}
    For $u_{1}, u_{2} \in \mathbb{Z}$, the matrix representing the action of tensoring with $\mathcal{O}\Big(\frac{u_{1}}{\lambda_{1}}F_{1}+ \frac{u_{2}}{\lambda_{2}} F_{2}\Big)$ is given by
    
    \begin{equation*}
    \left(
    (-) \otimes \mathcal{O} \Big(\frac{u_{1}}{\lambda_{1}}F_{1} + \frac{u_{2}}{\lambda_{2}}F_{2} \Big) \right)_{M} 
    = 
    \begin{bmatrix}
        1 & 0 & 0 & 0 \\
        u_{1} & 1 & 0 & 0 \\
        u_{2} & 0 & 1 & 0 \\
        u_{1}u_{2} & u_{2} & u_{1} & 1 \\
    \end{bmatrix}
    =
    \begin{bmatrix}
        1 & 0 \\
        u_{2} & 1 \\
    \end{bmatrix}
    \otimes
    \begin{bmatrix}
        1 & 0 \\
        u_{1} & 1 \\
    \end{bmatrix}.
    \end{equation*}
\end{rmk}

%% file: sections/section2/canonical_cover.tex
\subsection{Canonical covers} \label{canonical_cover}

For a variety whose canonical bundle has finite order, we can consider the canonical cover. We now summarize some properties of canonical covers useful for studying such surfaces.

\begin{definition}
    Let $X,Y$ be varieties. 
    Let $p_{X} \colon \tilde{X} \to X$, $p_{Y} \colon \tilde{Y} \to Y$ be canonical covers of $X,$ $Y$. Given a functor $\Phi \colon D(Y) \to D(X)$, a lift of $\Phi$ is a functor $\tilde{\Phi} \colon D(\tilde{Y}) \to D(\tilde{X})$ such that there exist isomorphisms of functors
    \[
    p_{X*} \circ \tilde{\Phi} \simeq \Phi \circ p_{Y*},
    \hspace{10pt}
    p_{X}^{*} \circ \Phi \simeq \tilde{\Phi} \circ p_{Y}^{*}.
    \]
\end{definition}

The following theorem is useful when considering autoequivalences on a bielliptic surface.

 \begin{thm}[{\cite[Theorem 4.5]{MR3713877}}] \label{LCC}
     Let $X,Y$ be varieties such that  $\ord(\omega_{X})$ and $\ord(\omega_{Y})$ are finite and coprime to $\Char (k)$. 
     Then for any equivalence \\ $\Phi \colon D(Y) \to D(X)$, there exists a lift of $\Phi$.
 \end{thm}
 \cite[Theorem 4.5]{MR3713877} is proved over $\mathbb{C}$. 
 However, the proof works over any algebraically closed field under the assumption that $\ord(\omega_{X})$ and $\ord(\omega_{Y})$ are coprime to $\Char (k)$ (see \cite[Proposition 2.1]{HonigsLieblichTirabassi2017arXiv1708.03409v1}). As a consequence of this theorem, we obtain the following.

 \begin{cor} \label{RIP}
    Let $X,Y$ be surfaces such that  $\ord(\omega_{X})$ and $\ord(\omega_{Y})$ are finite and coprime to $\Char (k)$. Take $\mathcal{E} \in D(\tilde{Y})$. Suppose that there exists an equivalence $\Phi \colon D(Y) \to D(X)$. Then $\rk(\Phi(p_{Y\ast}(\mathcal{E}))) \in \ord(\omega_{X})\mathbb{Z}$.
    In particular, for any point $y \in Y$, we have
    $\rk (\Phi (\mathcal{O}_{y})) \in \ord(\omega_{X})\mathbb{Z}$.
\end{cor}

\begin{proof}
    Note $\ch (p_{X\ast})(1,0,0) = (\ord(\omega_{X}),0,0)$ and $\rk(p_{X\ast}(\mathcal{F
    }))=0$ for any $\mathcal{F} \in \Coh(\tilde{X})$ with $\rk(\mathcal{F})=0$. 
    Thus for any $\mathcal{F} \in \Coh(\tilde{X
    })$, we obtain $\rk(p_{X\ast}(\mathcal{F})) \in \ord(\omega_{X}) \mathbb{Z}.$ This gives that
    \[
        \rk(\Phi(p_{Y\ast}(\mathcal{E})))
        = 
        \rk ( p_{X\ast}(\tilde{\Phi}(\mathcal{E}) )) \in \ord(\omega_{X})\mathbb{Z}. 
    \]
\end{proof}

\begin{rmk}
    Let $X$ be a bielliptic surface. 
    Note that $\ord (\omega_{X})$ is coprime to $\Char (k)$ (see \cref{table}).
\end{rmk}

We describe the canonical cover of a bielliptic surface. Let $X$ be a bielliptic surface. Then $X$ has one of the following invariants, depending on whether $\omega_{X}\simeq \mathcal{O}_X$ or not.

\begin{center}
    \begin{tabular}{c|ccc}
        $\omega_{X}$ & $\chi(\mathcal O_X)$ & $q(X)$ & $p_g(X)$ \\
        \hline
        trivial & $0$ & $2$ & $1$ \\
        non-trivial & $0$ & $1$ & $0$ \\
    \end{tabular}
\end{center}
If $\Char (k)=0$, then $\omega_{X} \not \simeq \mathcal{O}_{X}$ is always the case. On the other hand, $\omega_{X} \simeq \mathcal{O}_{X}$ occurs only if $\Char (k)=2$ or $3$.

Let $p \colon \tilde{X} \to X$ be the canonical cover of $X$ with $\omega_{X} \not \simeq \mathcal{O}_{X}$. 
Note that $K_{\tilde{X}}$ is trivial, and
\[
    \chi(\mathcal{O}_{\tilde{X}}) = \ord (\omega_{X}) \chi(\mathcal{O}_{X})=0
\]
(see \cite[Proposition 9.7]{badescu2001algebraic}). 
Then, by the classification of surfaces, the canonical cover $\tilde{X}$ of a bielliptic surface $X$ is 
\begin{itemize}
    \item 
    an abelian surface if $\Char(k) =0$, 
    \item  
    either an abelian surface, a bielliptic surface or a quasi-bielliptic surface if $\Char(k) >0$. 
\end{itemize}
In fact, the following lemma holds.

\begin{lem}
    Let $X$ be a bielliptic surface with $\omega_{X} \not \simeq \mathcal{O}_{X}$. Then the canonical cover of $X$ is not a quasi-bielliptic surface.
\end{lem}

\begin{proof}
    Suppose that $\tilde{X}$ is a quasi-bielliptic surface and let $C$ be a fiber of Albanese map of $\tilde{X}$, which is a rational curve with a cusp. Let $\tilde{C}$ be the normalization of $C$, which satisfies $\tilde{C} \simeq \mathbb{P}^{1}$. Then we have a surjective morphism  
    $\tilde{C} \to f_{1}^{-1}(q)$ for some $q \in \Alb (X)$, where $f_{1}^{-1}(q)$ is an elliptic curve. 
    On the other hand, we have $g(\tilde{C}) < g(f_{1}^{-1}(q))$.
    This is a contradiction.
\end{proof}

Therefore, the canonical cover of a bielliptic surface is either an abelian surface or bielliptic surface. This fact is one of the difficulties in treating the case $\Char (k)>0$.

%% file: sections/section2/relative_FM.tex
\subsection{Relative Fourier--Mukai transforms}
Let $f \colon X \to C$ be a relatively minimal elliptic surface.
We define
\[
    \lambda \coloneqq 
    \min \{ F \cdot D \mid D \mbox{ is the numerical class of an effective divisor} \},
\]
where $F \in \Num(X)$ denotes the numerical class of a smooth fiber of $f$. 
For any $\mathcal{E} \in D(X)$, we define the \textit{fiber degree} of $\mathcal{E}$ to be
\[
    d (\mathcal{E}) = c_{1}(\mathcal{E}) \cdot F.
\]
Suppose $a \in \mathbb{Z}_{>0}$, $b \in \mathbb{Z}$ with $a\lambda$ coprime to $b$. 
Then we can construct an elliptic surface $J_{X}(a,b)$ over $C$, whose fiber over a point $p \in C$ is canonically identified with a component of the moduli space of rank $a$, degree $b$, stable sheaves on the fiber $X_{p}$.
Bridgeland constructed an equivalence $\Phi_{\mathcal{P}} \colon D(J_{X}(a,b)) \to D(X)$ with a universal family $\mathcal{P}$.

\begin{thm}[{\cite[Theorem 5.3]{MR1629929}}] \label{RFM}
    Let $f \colon X \to C$ be a relatively minimal elliptic surface and let $F \in \Num (X)$ be the numerical class of a smooth fiber of $f$. Take a matrix
    \[
        \begin{bmatrix}
            c & a \\
            d & b \\
        \end{bmatrix} \in \SL_{2}(\mathbb{Z})
    \]
    such that $\lambda$ divides $d$ and $a >0$. Then there exists a sheaf $\mathcal{P}$ on $X \times J_{X}(a,b)$, flat and strongly simple over both factors such that for any point $(x,y) \in X \times J_{X}(a,b)$, $\mathcal{P}_{y}$ has Chern class $(0,aF,b)$ on $X$ and $\mathcal{P}_{x}$ has Chern class $(0,aF,-c)$ on $J_{X}(a,b)$. For any such sheaf $\mathcal{P}$, the resulting functor $\Phi \coloneqq \Phi^{\mathcal{P}}_{J_{X}(a,b) \to X} \colon D(J_{X}(a,b)) \to D(X)$, called relative Fourier--Mukai transform, is an equivalence and satisfies
    \[
        \begin{bmatrix}
            \rk(\Phi(\mathcal{E})) \\
            d(\Phi(\mathcal{E})) \\
        \end{bmatrix} = 
        \begin{bmatrix}
            c & a \\
            d & b \\
        \end{bmatrix}
        \begin{bmatrix}
            \rk(\mathcal{E}) \\
            d(\mathcal{E}) \\
        \end{bmatrix}
    \]
    for any object $\mathcal{E} \in D(J_{X}(a,b))$.
\end{thm}

Let $\Phi \colon D(J_{X}(a,b)) \to D(X)$ be a relative Fourier--Mukai transform obtained in \cref{RFM}. Then we define a $2 \times 2$ matrix
\[
    \Phi_{F} \coloneqq 
    \begin{bmatrix}
            c & a \\
            d & b \\
    \end{bmatrix}
    \in \SL_{2}(\mathbb{Z}).
\]

We construct a (non-standard) autoequivalence on a bielliptic surface.  
Recall that a bielliptic surface $X = (A \times B) / G$ has two natural elliptic fibrations
\[
    f_{1} \colon (A \times B) / G \to A /G, \quad 
    f_{2} \colon (A \times B) / G \to B /G.
\]
Take $a \in \mathbb{Z}_{>0}$ and $b \in \mathbb{Z}$ with $a\lambda$ coprime to $b$. Then the moduli space $J_{X}(a,b)$ associated to $f_{i}$ is isomorphic to $X$ for $i \in \{ 1,2 \}$. Indeed, note that $\lambda_i \in \{ 1,2,3,4,6 \}$, and then we can prove $X \simeq J_X(a,b)$ as in   \cite[Lemma 2.3]{uehara2024fouriermukaipartnersellipticruled}. 
Identifying $X$ and $J_X(a,b)$ by this isomorphism, the functor $\Phi : D(J_{X}(a,b)) \rightarrow D(X)$ gives an autoequivalence on $D(X)$. We also call this autoequivalence \emph{a relative Fourier--Mukai transform}.
This autoequivalence is not a standard autoequivalence. In fact, we have $\Phi^{A} (0,0,1) = (0,aF,b)$ for some $a \in \mathbb{Z}_{>0}$ and $b \in \mathbb{Z}$.
 
Let us consider the action of relative Fourier--Mukai transforms on the numerical Chow group $\mathbb{Z}\oplus \Num (X) \oplus \mathbb{Z}$ (see \cref{Matrix}).
The following proposition holds.

\begin{prop} \label{matrix}
    Let $X$ be a bielliptic surface.
    Let $\Phi_{1}$ be a relative Fourier--Mukai transform along $f_{1}$ and $\Phi_{2}$ be a relative Fourier--Mukai transform along $f_{2}$ with 
    \[
        \Phi_{1,F_{1}} = 
        \begin{bmatrix}
            c & a \\
            d & b \\
        \end{bmatrix}
        \in \SL_{2}(\mathbb{Z}),
        \Phi_{2,F_{2}} = 
        \begin{bmatrix}
            c^{\prime} & a^{\prime} \\
            d^{\prime} & b^{\prime} \\
        \end{bmatrix}
        \in \SL_{2}(\mathbb{Z}).
    \] 
    Then
    \[
         \Phi_{1,M} =
         \begin{bmatrix}
             c & a \lambda_{1} \\
             d/\lambda_{1} & b \\
         \end{bmatrix}
         \otimes
         \begin{bmatrix}
            1 & 0 \\
            s & 1 \\
         \end{bmatrix},
         \Phi_{2,M} =
         \begin{bmatrix}
             1 & 0 \\
             s^{\prime} & 1 \\
         \end{bmatrix}
         \otimes
         \begin{bmatrix}
            c^{\prime} & a^{\prime} \lambda_{2} \\
            d^{\prime}/\lambda_{2} & b^{\prime} \\
         \end{bmatrix}
    \]
    for some $s,s^{\prime} \in \mathbb{Z}$.
\end{prop}

\begin{proof}
    We only prove the statement for $\Phi_{1}$. 
    The case of $\Phi_{2}$ follows by a similar argument.
    Put $\Phi \coloneqq \Phi_{1}$ and $\lambda \coloneqq \lambda_{1}$. By \cref{RFM}, we have
    \[
        \Phi_{M}(0,0,0,1) =  (0, a\lambda, 0,b),
    \]
    \[
        \Phi_{M}(0,-a\lambda,0,c) = (0,0,0,1).
    \]
    As $\Phi_{M}$ is a homomorphism and 
    \[
        (0,1,0,0) = \frac{c}{a\lambda} (0,0,0,1) - \frac{1}{a\lambda} (0,-a\lambda,0,c),
    \]
    we obtain the following:
    \[
        \Phi_{M}(0,1,0,0) 
        = 
        \expar{ 0,c,0,\frac{bc - 1}{a\lambda}} = \expar{0,c,0,\frac{ad}{a\lambda}} = 
        \expar{ 0,c,0,\frac{d}{\lambda}}.
    \]

    As above, we can write
    \[
        \Phi_{M} = 
        \begin{bmatrix}
            x^{\prime} & 0 & x & 0 \\
            y^{\prime} & c & y & a\lambda \\
            z^{\prime} & 0 & z & 0 \\
            w^{\prime} & d/\lambda & w & b \\
        \end{bmatrix}
         \eqqcolon
        \begin{bmatrix}
            \bm{v_{1}} & \bm{v_{2}} & \bm{v_{3}} &\bm{v_{4}}\\
        \end{bmatrix}.
    \]
    Recall \cref{V14}, \cref{V23} and \cref{Vij}.
    As $\bm{v_{3}} \cdot \bm{v_{4}}=0$, we obtain
    \[
        -xb + za\lambda=0.
    \]
    Since $\gcd(a\lambda,b)=1$, it follows that $b$ divides $z$. Set $z \coloneqq bt$ for some $t \in \mathbb{Z}$. 
    Moreover, $\bm{v_{3}} \cdot \bm{v_{3}}=0$ implies
    \[
        -wa\lambda t + ybt = 0.
    \]
    If $t=0$ then $\bm{v_{2}} \cdot \bm{v_{3}}=0$, which is a contradiction. Therefore $t \ne 0$, and $yb = wa\lambda$. Hence $b$ divides $w$. Set $w \coloneqq bs$ for some $s \in \mathbb{Z}$. 

    Similarly, we can write $x^{\prime} \coloneqq c t^{\prime}, y^{\prime} \coloneqq c s^{\prime}$ for some $t^{\prime}, s^{\prime} \in \mathbb{Z}$. Then we obtain
    \[
        \Phi_{M} = 
        \begin{bmatrix}
            ct^{\prime} & 0 & a\lambda t & 0 \\
            cs^{\prime} & c & a\lambda s & a\lambda \\
            t^{\prime} d/\lambda  & 0 & bt & 0 \\
            s^{\prime} d/\lambda  & d/\lambda & bs & b \\
        \end{bmatrix}.
    \]
    Since $\bm{v_{2}} \cdot \bm{v_{3}}=1$ and $\bm{v_{1}} \cdot \bm{v_{4}}=-1$, we obtain $t=1$ and $t^{\prime}=1$ respectively.
    Hence we obtain $s=s^{\prime}$ as $\bm{v_{1}} \cdot \bm{v_{3}}=0$. Finally we have
    \[
        \Phi_{M} = 
        \begin{bmatrix}
            c & 0 & a\lambda  & 0 \\
            cs & c & a\lambda s & a\lambda \\
            d/\lambda & 0 & b & 0 \\
            s d/\lambda  & d/\lambda & bs & b \\
        \end{bmatrix}
        =
        \begin{bmatrix}
             c & a \lambda \\
             d/\lambda & b \\
         \end{bmatrix}
         \otimes
          \begin{bmatrix}
            1 & 0 \\
            s & 1 \\
        \end{bmatrix}.
    \]
\end{proof}

\begin{cor} \label{SMRFM}
    Let $\Phi_{1}$ and $\Phi_{2}$ be as in \cref{matrix}. Then there exist line bundles $\mathcal{L}_{1}, \mathcal{L}_{2}$ such that 
    \[
        (\Phi_{1} \circ ((-) \otimes \mathcal{L}_{1}))_{M}
        = (((-) \otimes \mathcal{L}_{1}) \circ \Phi_{1})_{M}
        =
        \begin{bmatrix}
                 c & a \lambda_{1} \\
                 d/\lambda_{1} & b \\
        \end{bmatrix}
        \otimes
        \begin{bmatrix}
                1 & 0 \\
                0 & 1 \\
        \end{bmatrix},
    \]
    and similarly,
    \[
        (\Phi_{2} \circ ((-) \otimes \mathcal{L}_{2}))_{M}
        = (((-) \otimes \mathcal{L}_{2}) \circ \Phi_{2})_{M}
        =
        \begin{bmatrix}
                 1 & 0 \\
                 0 & 1 \\
        \end{bmatrix}
        \otimes
        \begin{bmatrix}
                c^{\prime} & a^{\prime} \lambda_{2} \\
                d^{\prime}/\lambda_{2} & b^{\prime} \\
        \end{bmatrix}.
    \]
\end{cor}

\begin{proof}
    It follows immediately from \cref{LM} and \cref{matrix}.
\end{proof}

%% file: sections/section3/section3_main.tex
\section{The proof of \cref{MT1}} \label{section3}

\input{sections/section3/sub0}

\input{sections/section3/sub4}

\input{sections/section3/sub1}

\input{sections/section3/sub2}

%% file: sections/section3/sub0.tex
We outline the strategy for proving \cref{MT1}. First, we prove the following proposition in \cref{SUB1}.

\newcounter{breakpoint1}
\setcounter{breakpoint1}{\value{thm}}
\begin{prop}\label{Prop:rank}
    Let $X$ be a bielliptic surface and let $Y$ be a surface. Take a point $y \in Y$. Let $\Psi \colon D(Y) \to D(X)$ be an equivalence, and let $\mathcal{E} \coloneqq \Psi(\mathcal{O}_{y})$. Then there exists an autoequivalence $\Phi \in \Auteq D(X)$  which is a composition of relative Fourier--Mukai transforms such that $\rk(\Phi^{-1}(\mathcal{E}))=0$.
\end{prop}

Take a point $x$ on a bielliptic surface $X$ and $\Psi \in \Auteq D(X)$.
By \cref{Prop:rank}, we can find an equivalence $\Phi$ that is a composition of relative Fourier--Mukai transforms such that 
$\rk((\Phi^{-1} \circ \Psi) (\mathcal{O}_{x}) )=0$. 
Observe that
\[
    \chi(
    (\Phi^{-1} \circ \Psi) (\mathcal{O}_{x}) , 
    (\Phi^{-1} \circ \Psi) (\mathcal{O}_{x}) )  
    = \chi( \mathcal{O}_{x} ,\mathcal{O}_{x} )  
    = 0.
\]
Since $\Num (X)$ is generated by $(1/\lambda_{1}) F_{1}$ and $(1/\lambda_{2})  F_{2}$,  there exist $a \in \mathbb{Q}$, $b \in \mathbb{Z}$  and $i \in \{ 1,2\}$ such that 
\begin{equation}
    (\Phi^{-1} \circ \Psi)^{A}(0,0,1)=(0,a F_{i},b)
    \label{0aFb}.
\end{equation}

Next, we prove the following proposition in \cref{SUB2}.

\newcounter{breakpoint3}
\setcounter{breakpoint3}{\value{thm}}
\begin{prop} \label{Prop:frac}
    Let $X$ be a bielliptic surface. Take $\Psi \in \Auteq D(X)$. Suppose there exist $s,t \in \mathbb{Z}$ and $i \in \{1,2 \}$ such that $\Psi^{A}(0,0,1)= \expar{ 0,\frac{s}{\lambda_{i}} F_{i},t}.$ Then $s \in \lambda_{i}\mathbb{Z}$.
\end{prop}

By \cref{Prop:frac}, we know $a \in \mathbb{Z}$ in \cref{0aFb}. Moreover, we deduce that $a\lambda_{i}$ is coprime to $b$ by the next lemma.

\begin{lem} \label{COP}
    Let $X$ be a bielliptic surface. Take $\Psi \in \Auteq D(X)$. If 
    \[
        \Psi^{A}(0,0,1)
        = 
        \expar{r, \frac{s_{1}}{\lambda_{1}}F_{1} + \frac{s_{2}}{\lambda_{2}}F_{2}, t},
    \]
    then $r$, $s_{1}$, $s_{2}$ and $t$ have greatest common divisor 1.
\end{lem}

\begin{proof}
    Let 
    $\Psi^{A}(1,0,0)= (r^{\prime}, (s_1^{\prime} / \lambda_{1}) F_{1} + (s_2^{\prime}/\lambda_{2}) F_{2}, t^{\prime})$. Then we have
    \begin{eqnarray*}
        -1 
        &=& 
        \langle (0,0,1), (1,0,0) \rangle \\
        &=& 
        \exgen{
        \expar{ r, \frac{s_1}{\lambda_{1}}F_{1} + \frac{s_2}{\lambda_{2}}F_{2}, t},
        \expar{ r^{\prime}, \frac{s_1^{\prime}}{\lambda_{1}}F_{1} + \frac{s_2^{\prime}}{\lambda_{2}}F_{2}, t^{\prime}}
        } \\
        &=&
        -rt^{\prime} -tr^{\prime} + s_{1}s_{2}^{\prime} +s_{2}s_{1}^{\prime}.
    \end{eqnarray*}
    This proves the statement.
\end{proof}

Hence, by \cref{RFM}, there exists a relative Fourier--Mukai transform $\Phi^{\prime}$ along $f_{i}$ such that $\Phi^{\prime A}(0,0,1) = (0,aF_{i}, b)$. Then we obtain
\[
    (\Phi^{\prime -1} \circ \Phi^{-1} \circ \Psi)^{A} (0,0,1) = (0,0,1).
\]

Finally, we prove the following proposition in \cref{SUB2}.

\newcounter{breakpoint5}
\setcounter{breakpoint5}{\value{thm}}
\begin{prop} \label{Prop:sheaf}
    Let $X$ be a bielliptic surface and take $\Phi \in \Auteq D(X)$. Then $\Phi$ is a sheaf transform, i.e. for any point $x \in X$, $\Phi(\mathcal{O}_{x})$ is a shift of a sheaf. 
\end{prop}

By \cref{Prop:sheaf}, $\Phi^{\prime -1} \circ \Phi^{-1} \circ \Psi$ is a sheaf transform.
Thus there exists a point $y \in X$ such that $\Phi^{\prime -1} \circ \Phi^{-1} \circ \Psi (\mathcal{O}_{x}) \simeq \mathcal{O}_{y}$ up to shift.
Then we apply the following proposition.

\begin{prop}[{\cite[Corollary 5.23]{MR2244106}}] \label{JOS}
    Let $X$ and $Y$ be varieties. Suppose $\Phi \colon D(X) \simeq D(Y)$ is an equivalence such that for any $x \in X$ there exists $f(x) \in Y$ with
    \[
        \Phi(\mathcal{O}_{x}) \simeq \mathcal{O}_{f(x)}.
    \]
    Then $f \colon X \to Y$ defines an isomorphism and $\Phi$ is the composition of $f_{*}$ with the twist by some line bundle $M \in \rm{Pic}(Y)$, i.e.
    \[
        \Phi \simeq (M \otimes (-) ) \circ f_{*}.
    \]
\end{prop}

Therefore, for a given $\Psi \in \Auteq D(X)$, we can find relative 
Fourier--Mukai transforms $\Phi$, $\Phi^{\prime}$ such that $\Phi^{\prime -1} \circ \Phi^{-1} \circ \Psi$ is a standard autoequivalence. 
This completes the proof of \cref{MT1}.

%% file: sections/section3/sub4.tex
\subsection{Auxiliary results}

Here we collect the lemmas needed in the next section.

\begin{lem} \label{AFU5}
    Let $X$ be a surface. Let $\mathcal{E} \in \Coh(X)$ be a simple sheaf, i.e. $\Hom_{X} (\mathcal{E},\mathcal{E}) \simeq k$.
    Assume $\rk(\mathcal{E}) \neq 0$ and $\dim \Supp (\Tor(\mathcal{E})) \le 0$. 
    Then $\mathcal{E}$ is torsion-free.
\end{lem}

\begin{proof}
    Let $\mathcal{T} \coloneqq \Tor(\mathcal{E})$.
    Consider the following exact sequence:
    \[
        0 \to \Hom_{X}(\mathcal{E}, \mathcal{T})
        \to \Hom_{X}(\mathcal{E}, \mathcal{E})
        \to \Hom_{X}(\mathcal{E}, \mathcal{E} / \mathcal{T}).
    \]
    Since $\mathcal{E}$ is a simple sheaf, the map 
    $\Hom_{X}(\mathcal{E}, \mathcal{E})
    \to \Hom_{X}(\mathcal{E}, \mathcal{E} / \mathcal{T})$
    is injective, and hence $\Hom_{X}(\mathcal{E}, \mathcal{T})=0$.

    On the other hand, if $\mathcal{T} \neq 0$, then $\Hom_{X}(\mathcal{E}, \mathcal{T}) \neq 0$. Indeed, for $x \in \Supp (\mathcal{T})$, we can construct a non-zero morphism
    \[
        \mathcal{E} \twoheadrightarrow \mathcal{E} \otimes \mathcal{O}_{x} \twoheadrightarrow \mathcal{O}_{x} \hookrightarrow \mathcal{T}.
    \]
\end{proof}

\begin{lem} \label{AFU3}
    Let $X$ be a surface. Let $\mathcal{E}$ be a  torsion-free sheaf with $\ch (\mathcal{E}) = (1,0,0)$. Then $\mathcal{E}$ is an invertible sheaf.
\end{lem}

\begin{proof}
    Note that there exists an open set $U$ such that $\mathcal{E} |_{U}$ is locally free with $\codim (X \backslash U) = 2$. Since $\mathcal{E}$ is a torsion-free sheaf, there is a natural injection $\varphi \colon \mathcal{E} \hookrightarrow \mathcal{E}^{\vee \vee}$. Then $\varphi |_{U}$ is an isomorphism. 
    Note that $\mathcal{E}^{\vee \vee}$ is an invertible sheaf since $\rk (\mathcal{E}) =1$.

    Consider the exact sequence
    \[
    0 \to \mathcal{E} \to \mathcal{E}^{\vee \vee} \to \mathcal{E}^{\vee \vee} / \mathcal{E} \to 0.
    \]
    Since $\dim \Supp \mathcal{E}^{\vee \vee} / \mathcal{E} \le 0$, we have $\ch_{1} ( \mathcal{E}^{\vee \vee} / \mathcal{E} ) = 0$. 
    It follows that $\ch_{1} ( \mathcal{E}^{\vee \vee}) = 0$ since $\ch_{1} ( \mathcal{E}) = 0$. Thus we obtain $\ch (\mathcal{E}^{\vee \vee}) = (1,0,0)$ and $\ch (\mathcal{E}^{\vee \vee} / \mathcal{E}) = (0,0,0)$. 
    Hence $\mathcal{E}^{\vee \vee} / \mathcal{E} =0$ and $\mathcal{E} \simeq \mathcal{E}^{\vee \vee}$.  Since $\mathcal{E}^{\vee \vee}$ is invertible, so is  $\mathcal{E}$.
\end{proof}

\begin{lem} \label{NRT}
    Let $X$ be a bielliptic  
    with $\omega_{X} \simeq \mathcal{O}_{X}$.
    Then there exists no rigid torsion sheaf.
\end{lem}

\begin{proof}
    Assume that $\mathcal{T}$ is a rigid torsion sheaf,
    and set $\ch (\mathcal{T}) = (0,D,t)$. Then we have
    \begin{eqnarray*}
        - D^{2} 
        &=& - \langle (0,D,t),(0,D,t) \rangle \\
        &=& \chi (\mathcal{T},\mathcal{T}) \\
        &=& 2 \dim \Hom_{X} (\mathcal{T}, \mathcal{T}) - \dim \Ext_{X}^{1} (\mathcal{T}, \mathcal{T}) \\
        &=& 2 \dim \Hom_{X} (\mathcal{T}, \mathcal{T}) \\
        &\ge& 2.
    \end{eqnarray*}
    This gives $D^{2} \le -2$, which contradicts \cref{ED2}.
\end{proof}

\begin{lem} [{\cite[Lemma 2.2]{S_A_Kuleshov_1995}}] \label{mukai}
    Let $X$ be a surface.  
    Suppose that we have an exact sequence
    \[
        0 \to \mathcal{G}_{2} \to \mathcal{E} \to \mathcal{G}_{1} \to 0
    \]
    for coherent sheaves $\mathcal{E}$, $\mathcal{G}_{1}$ and $\mathcal{G}_{2}$ on $X$ satisfying 
    \[
        \Hom_{X} (\mathcal{G}_{2}, \mathcal{G}_{1}) 
        = \Ext_{X}^{2}(\mathcal{G}_{1}, \mathcal{G}_{2}) =0.
    \]
    Then
    \[
        \dim \Ext^{1}_{X}(\mathcal{E},\mathcal{E}) 
        \ge 
        \dim \Ext^{1}_{X}(\mathcal{G}_{1},\mathcal{G}_{1}) 
        + 
        \dim \Ext^{1}_{X}(\mathcal{G}_{2},\mathcal{G}_{2}).
    \]
\end{lem}

\begin{lem} \label{BIS}
    Let $X$ be a bielliptic 
    with $\omega_{X} \simeq \mathcal{O}_{X}$. 
    Take a simple sheaf $\mathcal{E} \in \Coh(X)$ with $\ch (\mathcal{E}) = (1,0,0).$ Then $\mathcal{E}$ is an invertible sheaf.
\end{lem}

\begin{proof}
    Let $\mathcal{T}$ be the torsion subsheaf of $\mathcal{E}$. If  
    $\dim \Supp \mathcal{T} \le 0$
    then $\mathcal{E}$ is a torsion-free sheaf by \cref{AFU5}. 
    Thus $\mathcal{E}$ is an invertible sheaf by \cref{AFU3}.

    We now show that $\dim \Supp \mathcal{T} \neq 1$.
    Suppose instead that $\dim \Supp \mathcal{T} = 1$. 
    As $\mathcal{E}$ is a simple sheaf, we have
    \begin{eqnarray*}
        0 &=& - \langle (1,0,0), (1,0,0) \rangle \\
        &=& \chi (\mathcal{E},\mathcal{E}) \\
        &=& 
        2 \dim \Ext_{X}^{0} (\mathcal{E},\mathcal{E})
        - \dim \Ext_{X}^{1} (\mathcal{E},\mathcal{E}) \\
        &=& 
        2 - \dim \Ext_{X}^{1} (\mathcal{E},\mathcal{E}).
    \end{eqnarray*}
    This gives $ \dim \Ext_{X}^{1}(\mathcal{E},\mathcal{E}) =2$.
    On the other hand, we know that 
    $\tr \circ \I = \rk (\mathcal{E}/ \mathcal{T} ) \cdot \id = \id$ is injective. Here, $\I$ and $\tr$ are defined in \cite[Lemma 10.1.3]{huybrechts2010geometry}.
    \[
        \I \colon H^{1}(X, \mathcal{O}_{X}) 
        \to 
        \Ext_{X}^{1} (\mathcal{E} / \mathcal{T} , \mathcal{E} / \mathcal{T} ) ,
    \]
    \[
        \tr \colon 
        \Ext_{X}^{1} (\mathcal{E} / \mathcal{T} , \mathcal{E}/ \mathcal{T} ) 
        \to H^{1}(X, \mathcal{O}_{X}).
    \]
    This gives
    \[
        \dim H^{1}(X, \mathcal{O}_{X}) 
        \le 
        \dim 
        \Ext_{X}^{1} (\mathcal{E} / \mathcal{T}, \mathcal{E}/ \mathcal{T}).
    \]
    Thus we have 
    \[
        \dim \Ext_{X}^{1}(\mathcal{E} / \mathcal{T},\mathcal{E} / \mathcal{T}) \ge 2.
    \]
    Moreover, we have
    \[
        \dim \Ext_{X}^{1} (\mathcal{E},\mathcal{E})
        \ge 
        \dim \Ext_{X}^{1} (\mathcal{E} / \mathcal{T},\mathcal{E} / \mathcal{T}) + \dim \Ext_{X}^{1} (\mathcal{T}, \mathcal{T})
    \]
    by \cref{mukai}. 
    Therefore we have $\dim \Ext_{X}^{1} (\mathcal{T}, \mathcal{T})=0.$
    This contradicts to \cref{NRT}. 
\end{proof}

\begin{lem} \label{AVA}
    Let $X$ be a bielliptic surface 
    with $\omega_{X} \simeq \mathcal{O}_{X}$ and $\Phi \in \Auteq D(X)$ be an autoequivalence.
    Then $\Phi^{A} (0,0,1) \neq (1,0,0)$.
\end{lem}

\begin{proof}
    Assume that there exists an autoequivalence $\Phi_{\mathcal{P}} \in \Auteq D(X)$ such that $\Phi_{v(\mathcal{P})}^{A} (0,0,1) = (1,0,0)$. 
    Let $\mathcal{U}$ be a Poincar\'{e} bundle on $\underline{\Pic}^0(X) \times X$. 
    By \cref{BIS}, If necessary, we apply a shift to $\Phi$, $\Phi_{\mathcal{P}}(\mathcal{O}_{x})$ is an invertible sheaf for any $x \in X$. 
    Therefore, by the universal property of Picard variety, there exists a morphism $\varphi \colon X \to \underline{\Pic}^0(X)$ such that
    \[
        \mathcal{P} =
        (\varphi \times \id)^{\ast} \mathcal{U}.
    \]
    Let $\psi \colon \underline{\Pic}^0(X)_{\red}  \to \underline{\Pic}^0(X)$ be the closed immersion
    and $\varphi^{\prime} \colon X \to \underline{\Pic}^0(X)_{\red}$ be an induced morphism by $\varphi$.
    Set $\mathcal{U}^{\prime} \coloneqq (\psi \times id)^{\ast}\mathcal{U}$.
    It follows that 
    \[
        \mathcal{P} =
        (\varphi \times \id)^{\ast} \mathcal{U} 
        =
        ((\psi \circ \varphi^{\prime}) \times \id)^{\ast} \mathcal{U}
        =
        (\varphi^{\prime} \times \id)^{\ast} \mathcal{U}^{\prime}.
    \]
    Then
    $\Phi_{\mathcal{P}} =
    \Phi_{\mathcal{U}^{\prime}} \circ \mathbb{R} \varphi_{\ast}^{\prime}$.

    Note that $\underline{\Pic}^0(X)_{\red}$ is an elliptic curve.
    For any integral functors $\Phi \colon D(X) \to D(\underline{\Pic}^0(X)_{\red}), \Psi \colon D(\underline{\Pic}^0(X)_{\red}) \to D(X)$, the following morphisms are induced as in \cref{Matrix}:
    \[
        \Phi_{M} \colon \mathbb{Z} \oplus \Num (X) \oplus  \mathbb{Z} \simeq \mathbb{Z}^{4} \to \mathbb{Z} \oplus \mathbb{Z},
    \]
    \[
        \Psi_{M} \colon \mathbb{Z} \oplus \mathbb{Z} \to \mathbb{Z} \oplus \Num (X) \oplus  \mathbb{Z} \simeq \mathbb{Z}^{4}.
    \]
    We have $(\mathbb{R} \varphi_{\ast}^{\prime})_{M} \in M_{2,4}(\mathbb{Z}), \Phi_{\mathcal{U}^{\prime}, M} \in M_{4,2}(\mathbb{Z})$ and
    \[
        \Phi_{\mathcal{Q},M} = \Phi_{\mathcal{U}^{\prime}, M} (\mathbb{R} \varphi_{\ast}^{\prime})_{M}.
    \]
    This contradicts the fact that $\Phi_{\mathcal{P},M}$ is automorphism on $\mathbb{Z}^{4}$.
\end{proof}

%% file: sections/section3/sub1.tex
\subsection{The proof of \cref{Prop:rank}} \label{SUB1}
Let $(a,b)$ denote the greatest common divisor of $a,b \in \mathbb{Z}$. 
We will frequently use the following elementary fact.

\begin{lem} 
    Let $a,b,n \in \mathbb{Z}$ and set $h \coloneqq (a,b).$ Then the following are equivalent:
    \begin{enumerate}
        \item $(na,b) = (a,b)$.
        \item $\expar{ n,\frac{b}{h}} = 1$.
    \end{enumerate}
\end{lem}

\begin{proof}
Take $a,b \in \mathbb{Z}$, and set $a^{\prime} \coloneqq a/h,$ $b^{\prime} \coloneqq b/h$. Note that
        \[
            (a,b) = (a^{\prime}h,b^{\prime}h) = h(a^{\prime},b^{\prime}),
        \]
        \[
            (na,b) = (na^{\prime}h,b^{\prime}h) = h(na^{\prime},b^{\prime}).
        \]
Thus the equality $(na,b) = (a,b)$ is equivalent to $(na^{\prime},b^{\prime}) = (a^{\prime},b^{\prime})$. Thus we can reduce to the case $(a,b)=1$. 
In this case, the assertion is clear.
\end{proof}

\begin{lem} \label{ERD}
    Let $f \colon X \rightarrow C$ be an elliptic surface and let $F \in \rm{Num}(X)$ denote the numerical class of a fiber of $f$. For any object $\mathcal{E} \in D(X)$, the following statements hold:
    \begin{enumerate}
        \item 
        If 
        $(\lambda \rk (\mathcal{E}),d(\mathcal{E})) = (\rk (\mathcal{E}),d(\mathcal{E}))$,
        then there exists a relative Fourier--Mukai transform $\Phi$ along $f$ such that $\rk (\Phi^{-1}(\mathcal{E}))=0$.
        \item 
        If 
        $(\lambda \rk (\mathcal{E}),d(\mathcal{E})) = \lambda (\rk (\mathcal{E}),d(\mathcal{E}))$, then there exists a relative Fourier--Mukai transform $\Phi$ along $f$ such that $d(\Phi^{-1}(\mathcal{E}))=0$.
    \end{enumerate}
\end{lem}

\begin{proof}
    Set $r\coloneqq \rk (\mathcal{E})$, $d \coloneqq d(\mathcal{E}),$ $h \coloneqq (r,d)$.
    
    (1) 
    After applying a shift functor if necessary, we may assume $r>0$. There exist integers $x,y \in \mathbb{Z}$ such that 
    \[
        x \frac{d}{h}-y\frac{r}{h}=1.
    \]
    Then, for any $n \in \mathbb{Z}$, we have
    \[
        \expar{x-n\frac{r}{h}} \frac{d}{h} - \expar{ y-n\frac{d}{h}}\frac{r}{h} =1.
    \]
    Since the assumption in (1) is equivalent to $(\lambda,d/h) = 1$, 
    there exists an integer $n \in \mathbb{Z}$ such that $y-nd/h \in \lambda\mathbb{Z}$. Then, by \cref{RFM}, there exists a relative Fourier--Mukai transform $\Phi$ along $f$ with 
    \[
        \Phi_{F}= 
        \begin{bmatrix}
        x-nr/h & r/h \\
        y-nd/h & d/h \\
        \end{bmatrix}.
    \]
    Consequently,
    \[
        \Phi_{F}^{-1} = 
        \begin{bmatrix}
         d/h & -r/h \\
        -y+nd/h & x-nr/h \\
        \end{bmatrix}.
    \]
    Clearly, $\Phi$ satisfies the desired condition.

    (2) 
    There exist integers $x \in \mathbb{Z}$ and $y \in \mathbb{Z}_{>0}$ such that $$x\frac{r}{h}-y\frac{d}{h}=1.$$ Note that
    \[ 
        (\lambda r,d) = \lambda (r,d) \Leftrightarrow \lambda \left| \frac{d}{h} \right. .
    \]
    Then, by \cref{RFM}, there exists a relative Fourier--Mukai transform $\Phi$ along $f$ with 
    \[
        \Phi_{F}= 
        \begin{bmatrix}
        r/h & y \\
        d/h & x \\
        \end{bmatrix}.
    \]
    In particular, we have
    \[
        \Phi_{F}^{-1} = 
        \begin{bmatrix}
        x & -y \\
        -d/h & r/h \\
        \end{bmatrix}.
    \]
    Clearly, $\Phi$ satisfies the desired condition.
\end{proof}

\newcounter{breakpoint2}
\setcounter{breakpoint2}{\value{thm}}
\counterwithin{thm}{section}
\setcounter{thm}{\value{breakpoint1}}
\begin{prop}
    \textup{(}Restated\textup{)} Let $X$ be a bielliptic surface and $Y$ be a surface. Take a point $y \in Y$. Let $\Psi \colon D(Y) \to D(X)$ be an equivalence, and set $\mathcal{E} \coloneqq \Psi(\mathcal{O}_{y})$. Then there exists an autoequivalence $\Phi \in \Auteq D(X)$  which is a composition of relative Fourier--Mukai transforms such that $\rk (\Phi^{-1}(\mathcal{E}))=0$.
\end{prop}

\begin{proof}
    For any object $\mathcal{E} \in D(X)$ and $i \in \{1,2\}$, let the fiber degree of $\mathcal{E}$ be denoted by 
    \[
    d_{i} (\mathcal{E} ) = c_{1}(\mathcal{E}) \cdot F_{i}.
    \]

    \textit{Step 1.}
    We apply a relative Fourier--Mukai transform along $f_{1}$. 
    In the case $\lambda_{1}=1$,  the statement directly follows from \cref{ERD} (1).
    Thus we may assume $\lambda_{1} \neq 1$, and by \cref{table} we see that $\lambda_{1}$ is prime and $\lambda_{2} \neq 1$. 
    If 
    $(\lambda_{1} \rk (\mathcal{E}),d_{1}(\mathcal{E})) = (\rk (\mathcal{E}),d_{1}(\mathcal{E}))$, 
    the assertion follows again by \cref{ERD} (1). 
    
    Assume  $(\lambda_{1} \rk (\mathcal{E}),d_{1}(\mathcal{E})) \neq (\rk (\mathcal{E}),d_{1}(\mathcal{E}))$,
    which is equivalent to \ \\
    $\lambda_{1}(\rk (\mathcal{E}),d_{1}(\mathcal{E}))=(\lambda_{1} \rk (\mathcal{E}),d_{1}(\mathcal{E}))$
    as $\lambda_{1}$ is prime. By \cref{ERD} (2), there exists a relative Fourier--Mukai transform  $\Phi_{1}$ along $f_{1}$ such that $d_{1}(\Phi_{1}^{-1}(\mathcal{E}))=0$. 
    We set 
    $\mathcal{E}^{\prime} \coloneqq \Phi_{1}^{-1}(\mathcal{E})$
    and
    $r^{\prime} \coloneqq \rk (\mathcal{E}^{\prime})$.

    \textit{Step 2.}
    We apply a relative Fourier--Mukai transform along $f_{2}$. 
    Assume, for a contradiction, that 
    $(\lambda_{2} r^{\prime},d_{2}(\mathcal{E}^{\prime})) \neq 
    (r^{\prime},d_{2}(\mathcal{E}^{\prime}))$. 
    Then there exists a prime factor $p$ of $\lambda_{2}$ such that 
    $p$ divides 
    $d_{2}(\mathcal{E}^{\prime})/(r^{\prime},d_{2}(\mathcal{E}^{\prime}))$.
    
    We claim that $r^{\prime} \in \lambda_{2} \mathbb{Z}$.
    If the type of $X$ is not $(2, \mu_{2})$, this holds since $r^{\prime} \in \ord(\omega_{X}) \mathbb{Z}$ by \cref{RIP}, and $\ord(\omega_{X}) = \lambda_{2}$ by \cref{table}. 
    Assume the type of $X$ is $(2, \mu_{2})$ and $r^{\prime} \not\in \lambda_{2} \mathbb{Z}$. 
    By Step 1 and 
    \[
        0 = \langle (0,0,1) , (0,0,1) \rangle = \langle\ch(\mathcal{E}^{\prime}) ,\ch(\mathcal{E}^{\prime}) \rangle,
    \]
    we can write $\ch(\mathcal{E}^{\prime}) = (r^{\prime}, (s/\lambda_{1})F_{1} , 0)$ with some integer $s \in \mathbb{Z}$. 
    By \cref{COP}, there exist integers $x, y \in \mathbb{Z}$ such that $yr^{\prime} + xs =1$. 
    Note that $\lambda_{2}$ is prime.
    As $r^{\prime} \not\in \lambda_{2} \mathbb{Z}$, there exists an integer $n \in \mathbb{Z}$ such that $x + nr^{\prime} \in \lambda_{2} \mathbb{Z}_{>0}$. 
    Therefore, we have a relative Fourier--Mukai transform $\Phi_{2}$ along $f_{2}$ with
    \[
        \Phi_{2,F_{2}} =
        \begin{bmatrix}
            y -sn & (x + nr^{\prime})/\lambda_{2} \\
            -s\lambda_{2} & r^{\prime} \\
        \end{bmatrix}.
    \]
    Then we can calculate
    \[
        \begin{bmatrix}
            \rk (\Phi_{2}(\mathcal{E}^{\prime})) \\
            d_{2} (\Phi_{2}(\mathcal{E}^{\prime})) \\
        \end{bmatrix}
        =
        \begin{bmatrix}
            y -sn & (x + nr^{\prime})/\lambda_{2} \\
            -s\lambda_{2} & r^{\prime} \\
        \end{bmatrix}
        \begin{bmatrix}
            r^{\prime} \\
            s\lambda_{2} \\
        \end{bmatrix}
        =
        \begin{bmatrix}
            1 \\
            0 \\
        \end{bmatrix}.
    \]
    Thus we obtain 
    $( \Phi_{2} \circ \Phi_{1}^{-1} \circ \Psi)^{A}(0,0,1) 
    = \Phi_{2}^{A} (\ch(\mathcal{E}^{\prime})) = (1,(u/\lambda_{2})F_{2},0)$ with some $u \in \mathbb{Z}$.
    Let $\Phi_{3} \coloneqq (-) \otimes \mathcal{O}((-u/\lambda_{2})F_{2}) \in \Auteq D(X)$ be line bundle tensoring.
    Then we have 
    $(\Phi_{3} \circ \Phi_{2} \circ \Phi_{1}^{-1} \circ \Psi)^{A}(0,0,1) 
    = (1,0,0)$
    which contradicts \cref{AVA}. 
    Thus, the claim holds.

    We can write 
    \[
       \ch(\mathcal{E}^{\prime}) = \expar{ n \lambda_{2}, \frac{s}{\lambda_{1}}F_{1} , 0}
    \]
    with $n,s \in \mathbb{Z}$. Thus we have
    \[
        \frac{d_{2}(\mathcal{E}^{\prime})}{(\rk (\mathcal{E}^{\prime}),d_{2}(\mathcal{E}^{\prime}))} 
        = \frac{s\lambda_{2}}{(n \lambda_{2},s\lambda_{2})} 
        = \frac{s\lambda_{2}}{\lambda_{2}(n,s)}
        = \frac{s}{(n,s)}.
    \]
    In particular, this implies that $p \mid s$.
    Since $\lambda_2$ and $s$ have a common divisor $p$, this contradicts \cref{COP}. 
    Therefore  
    $(\lambda_{2} \rk (\mathcal{E}^{\prime}),d_{2}(\mathcal{E}^{\prime})) = 
    (\rk (\mathcal{E}^{\prime}),d_{2}(\mathcal{E}^{\prime}))$, and hence the statement follows by \cref{ERD} (1). 
\end{proof}

\counterwithin{thm}{section}
\setcounter{thm}{\value{breakpoint2}}

%% file: sections/section3/sub2.tex
\subsection{The proof of \cref{Prop:frac} and \cref{Prop:sheaf}} \label{SUB2}

\begin{lem} \label{rigid}
    Let $X$ be a bielliptic surface with $\omega_{X} \simeq \mathcal{O}_{X}$. Then there does not exist a rigid sheaf.
\end{lem} 

\begin{proof}
    We argue by contradiction and suppose that there exists a rigid coherent sheaf $\mathcal{F}$ on $X$. 
    Note that $\mathcal{F}$ is not a torsion sheaf by \cref{NRT}.
    We apply \cref{mukai} to the short exact sequence
    \[
        0 \to \Tor (\mathcal{F}) \to \mathcal{F} \to \mathcal{F} / \Tor (\mathcal{F}) \to 0
    \]
    and we know that $\mathcal{F} / \Tor (\mathcal{F})$ is also a rigid coherent sheaf.
    Let
    \[
        0=\mathcal{G}_{0} \subset \mathcal{G}_{1} \subset \cdots \subset \mathcal{G}_{l} = \mathcal{F} / \Tor (\mathcal{F})
    \]
    be the Harder--Narasimhan filtration of $\mathcal{F} / \Tor (\mathcal{F})$. By \cite[Lemma 1.3.3]{huybrechts2010geometry} we know $\Hom_{X} (\mathcal{G}_{l-1}, \mathcal{G}_{l}/\mathcal{G}_{l-1}) =0$. 
    We apply \cref{mukai} to the short exact sequence
    \[
        0 \to \mathcal{G}_{l-1} \to \mathcal{G}_{l} \to \mathcal{G}_{l}/\mathcal{G}_{l-1} \to 0
    \]
    and get a rigid $\mu$-semistable sheaf $\mathcal{E} \coloneqq \mathcal{G}_{l}/\mathcal{G}_{l-1}$.
    Then we have 
    \[
        \Delta(\mathcal{E}) \coloneqq\ch_{1}(\mathcal{E})^{2} - 2 \ch_{0}(\mathcal{E})\ch_{2}(\mathcal{E}) \ge 0
    \]
    by \cite[Theorem 1.1]{10.1093/imrn/rnac260}. On the other hand, we compute
    \begin{align*}
        \Delta(\mathcal{E})
        &= - \chi(\mathcal{E}, \mathcal{E}) \tag*{(Riemann--Roch)} \\
        &= - 2 \dim \Hom_X(\mathcal{E}, \mathcal{E})
        + \dim \Ext^1_X(\mathcal{E}, \mathcal{E}) \\
        &\le -2 + \dim \Ext^1_X(\mathcal{E}, \mathcal{E}).
    \end{align*}
    Therefore we obtain $\dim \Ext^{1}_{X}(\mathcal{E}, \mathcal{E}) \ge 2$. This is a contradiction by the rigidity of $\mathcal{E}$.
\end{proof}

\begin{lem} \label{Prop:not-simple-sheaf}
    Let $X$ be a bielliptic surface with $\omega_{X} \simeq$ $\mathcal{O}_{X}$. Take any $\Phi \in \Auteq D(X)$. Then $\Phi$ is a sheaf transform. 
\end{lem}

\begin{proof}
    Take any $x \in X$. By \cref{rigid}, we have
    \[
        \dim \Ext_{X}^{1} (\mathcal{H}^{i}(\Phi(\mathcal{O}_{x})) , \mathcal{H}^{i}(\Phi(\mathcal{O}_{x}))) \ge 2
    \]
    for all $i \in \mathbb{Z}$ with $\mathcal{H}^{i}(\Phi(\mathcal{O}_{x})) \neq 0$.
    On the other hand, for $\mathcal{E}, \mathcal{F} \in D(X)$, we have the following spectral sequence (see \cite[I\hspace{-1.2pt}V \S2]{gelfand2013methods}):
    \[
        E_{2}^{p,q} = \bigoplus_{i} \Ext_{X}^{p} (\mathcal{H}^{i}(\mathcal{E}), \mathcal{H}^{i+q}(\mathcal{F}) ) \Rightarrow \Ext_{X}^{p+q} (\mathcal{E},\mathcal{F}).
    \]
    Since the $E_{2}^{1,0}$ term survives to infinity, we obtain
    \begin{eqnarray*}
         \sum_{i} \dim \Ext_{X}^{1} (\mathcal{H}^{i}(\Phi(\mathcal{O}_{x})) , \mathcal{H}^{i}(\Phi(\mathcal{O}_{x}))) 
        &\le&
        \dim \Ext_{X}^{1} (\Phi(\mathcal{O}_{x}) , \Phi(\mathcal{O}_{x})) \\
        &=& \dim \Ext_{X}^{1} (\mathcal{O}_{x} , \mathcal{O}_{x}) \\
        &=& 2.
    \end{eqnarray*}
    Hence the assertion follows.
\end{proof}

\begin{lem} \label{frac_not_simple}
    Let $X$ be a bielliptic surface with $\omega_{X} \simeq \mathcal{O}_{X}$. Take $\Psi \in \Auteq D(X)$. Suppose there exist $s,t \in \mathbb{Z}$ and $i \in \{1,2 \}$ such that $\Psi^{A}(0,0,1)= \expar{ 0,\frac{s}{\lambda_{i}} F_{i},t}.$ Then $s \in \lambda_{i}\mathbb{Z}$.
\end{lem}

\begin{proof}
    If necessary applying the shift functor, we conclude that $\Psi(\mathcal{O}_{x})$ is a sheaf by \cref{Prop:not-simple-sheaf}.
    Then $D \coloneqq \ch_{1}(\Psi(\mathcal{O}_{x}))$ is a connected effective divisor,
    since $\Psi(\mathcal{O}_{x})$ is a simple sheaf of dimension one.
    Suppose $s \notin \lambda_{i}\mathbb{Z}$. Then we deduce $i = 2$ because $f_{1}$ has no multiple fibers. 
    Hence the support of $D$ is contained in a multiple fiber of $f_{2}$.
    Set $Z \coloneqq \cup \Supp F_{i}$, where $F_{i}$ is a multiple fiber of $f_{2}$. We have $\Supp \Psi (\mathcal{O}_{x}) \subset Z$ for any $x \in X$. 
    Let $\mathcal{P}$ be a Fourier--Mukai kernel of $\Psi$.
    Then, we obtain $\Supp (\mathcal{P}) \subset X \times Z$. Thus Supp $\Supp \Psi (\mathcal{E}) \subset Z$ for any $\mathcal{E} \in D(X)$. 
    This contradicts the fact that $\Psi$ is an autoequivalence.
\end{proof}

\newcounter{breakpoint6}
\setcounter{breakpoint6}{\value{thm}}
\counterwithin{thm}{section}
\setcounter{thm}{\value{breakpoint5}}

\begin{prop} \label{SCS}
    \textup{(}Restated\textup{)} 
    Let $X$ be a bielliptic surface and take \\
    $\Phi \in \Auteq D(X)$. Then $\Phi$ is a sheaf transform. 
\end{prop}

\begin{proof}
    It suffices to prove the statement in the case $\omega_{X} \not \simeq \mathcal{O}_{X}$ by \cref{Prop:not-simple-sheaf}.
    In \cref{canonical_cover}, we have seen that the canonical cover of $X$ is an abelian surface or a bielliptic surface with the canonical sheaf trivial.
    By \cref{LCC}, there exists $\tilde{\Phi} \in \Auteq D(\tilde{X})$ which is a lift of $\Phi$. 
    Then $\tilde{\Phi}$ is a sheaf transform by \cref{SCS} and the fact that any equivalences on abelian varieties are sheaf transforms by \cite[Proposition 3.2]{MR1921811}.  Hence $\Phi$ is a sheaf transform by the commutativity of the lifting diagram.
\end{proof}

\counterwithin{thm}{section}
\setcounter{thm}{\value{breakpoint6}}

\newcounter{breakpoint4}
\setcounter{breakpoint4}{\value{thm}}
\counterwithin{thm}{section}
\setcounter{thm}{\value{breakpoint3}}

\begin{prop}
    \textup{(}Restated\textup{)} 
    Let $X$ be a bielliptic surface and take $\Psi \in \Auteq D(X)$. Suppose there exist $s,t \in \mathbb{Z}$ and $i \in \{1,2 \}$ such that $\Psi^{A}(0,0,1)= \expar{ 0,\frac{s}{\lambda_{i}} F_{i},t}.$ Then $s \in \lambda_{i}\mathbb{Z}$.
\end{prop}

\begin{proof}
    It is enough to prove the statement in the case of $\omega_{X} \not \simeq \mathcal{O}_{X}$.
    However, it can be shown in the same way as \cref{frac_not_simple} by using \cref{Prop:sheaf}.
\end{proof}

\counterwithin{thm}{section}
\setcounter{thm}{\value{breakpoint4}}

%% file: sections/section4/section4_main.tex
\section{Consequences of \cref{section3}}

\input{sections/section4/app1}
\input{sections/section4/app0}

%% file: sections/section4/app1.tex
Let $\Gamma (\lambda_{i})$ denote the subgroup of $\SL_{2}(\mathbb{Z})$ as
\[
\Gamma (\lambda_{i}) \coloneqq 
\inset{
\begin{bmatrix}
    c & a \\
    d & b \\
\end{bmatrix} 
\in 
\SL_{2}(\mathbb{Z}) }{ a \in \lambda_{i} \mathbb{Z} }.
\]
As a consequence of \cref{SUB2}, we obtain the following theorem.

\begin{thm} \label{ESC}
    Let $X$ be a bielliptic surface over $k$ of arbitrary characteristic. Then we have a short exact sequence
    \[
        1 
        \to  \Aut (X) \times \mathbb{Z} [2]
        \to \Auteq D(X) 
        \xrightarrow{\pi} \{ A_{1} \otimes A_{2} \in \GL_{4}(\mathbb{Z}) |
        A_{i} \in \Gamma (\lambda_{i}) \}
        \to 1,
    \]
    where $\pi$ is defined by $\pi (\Phi) = \Phi_{M}$
    for $\Phi \in \Auteq D(X)$.
\end{thm}

\begin{proof}
    It is clear that $\Aut (X) \times \mathbb{Z} [2] \subset \Ker \pi$, because automorphisms preserve effective divisors and cannot exchange the fibers of the two elliptic fibrations as one has multiple fibers and the other does not.
    Take $\Phi \in \Ker \pi$. 
    Since $\Phi^{A}$ maps $(0,0,1)$ to $(0,0,1)$  
    and $\Phi$ is a sheaf transform by \cref{Prop:sheaf}, 
    $\Phi$ is a standard autoequivalence. 
    By \cref{LM} and \cref{NUM}, 
    non-trivial line bundle tensoring acts on the numerical Chow group non-trivially.
    Thus it follows that $\Phi \in  \Aut (X) \times \mathbb{Z} [2]$.

    Next, we prove the surjectivity of $\pi$. Take $A_{1} \otimes A_{2} \in \GL_{4}(\mathbb{Z})$ with $A_{i} \in \Gamma (\lambda_{i})$. Let $A_{1} \coloneqq (x_{ij})$ and $A_{2} \coloneqq (y_{ij})$. Then we can find a relative Fourier--Mukai transform $\Phi_{1}$ along $f_{1}$ such that
    \[
        \Phi_{1, F_{1}} =
        \begin{bmatrix}
            x_{11} & x_{12} / \lambda_{1} \\
            x_{21}\lambda_{1} & x_{22} \\
        \end{bmatrix}.
    \]
    We can also find a relative Fourier--Mukai transform $\Phi_{2}$ along $f_{2}$ such that
    \[
        \Phi_{2, F_{2}} =
        \begin{bmatrix}
            y_{11} & y_{12} / \lambda_{2} \\
            y_{21}\lambda_{2} & y_{22} \\
        \end{bmatrix}.
    \]
    By \cref{SMRFM}, there exist line bundles $\mathcal{L}_{1}, \mathcal{L}_{2}$ such that
    \[
        (\Phi_{1} \circ ((-) \otimes \mathcal{L}_{1}))_{M}
        =
        \begin{bmatrix}
            x_{11} & x_{12} \\
            x_{21} & x_{22} \\
        \end{bmatrix}
        \otimes
        \begin{bmatrix}
                1 & 0 \\
                0 & 1 \\
        \end{bmatrix}
    \]
    and
    \[
        (\Phi_{2} \circ ((-) \otimes \mathcal{L}_{2}))_{M}
        =
        \begin{bmatrix}
                 1 & 0 \\
                 0 & 1 \\
        \end{bmatrix}
        \otimes
        \begin{bmatrix}
                y_{11} & y_{12} \\
                y_{21} & y_{22} \\
        \end{bmatrix}.
    \]
    Put $\Phi^{\prime}_{i} \coloneqq \Phi_{i} \circ ((-) \otimes \mathcal{L}_{i})$ for each $i \in \{ 1,2 \}$.
    Then $\pi (\Phi_{1}^{\prime} \circ \Phi_{2}^{\prime}) = A_{1} \otimes A_{2}$.
\end{proof}

%% file: sections/section4/app0.tex
\newcounter{breakpoint8}
\setcounter{breakpoint8}{\value{mthm}}
\setcounter{mthm}{\value{breakpoint7}}

\begin{mthm}  \textup{(}Restated\textup{)} 
    A bielliptic surface $X$ over an algebraically closed field $k$ of arbitrary characteristic has no non-trivial Fourier--Mukai partners.
\end{mthm}

\setcounter{mthm}{\value{breakpoint8}}

\begin{proof}
    Let $Y$ be a surface. 
    Suppose that there exists an equivalence $\Psi \colon D(Y) \to D(X)$, and take $y \in Y$. By \cref{Prop:rank}, \cref{Prop:frac} and \cref{COP}, there exists $\Phi \in \Auteq D(X)$ which is a composition of relative Fourier--Mukai transforms and a shift such that
    \[
        (\Phi \circ \Psi)^{A} (0,0,1) = (0,0,1).
    \]
    We  know that $\Phi \circ \Psi$ is a sheaf transform by \cref{Prop:sheaf}. 
    Then $(\Phi \circ \Psi)(\mathcal{O}_{y})$ is a sheaf whose Chern character is $(0,0,1)$. Thus there exists a point $\varphi(y) \in X$ such that
    \[
        (\Phi \circ \Psi) (\mathcal{O}_{y}) \simeq \mathcal{O}_{\varphi(y)}.
    \]
    By \cref{JOS}, it follows that $\varphi$ is an isomorphism. Thus $X \simeq Y$.
\end{proof}

%% file: main.bbl
\providecommand{\bysame}{\leavevmode\hbox to3em{\hrulefill}\thinspace}
\providecommand{\MR}{\relax\ifhmode\unskip\space\fi MR }
\providecommand{\MRhref}[2]{%
  \href{http://www.ams.org/mathscinet-getitem?mr=#1}{#2}
}
\providecommand{\href}[2]{#2}
\begin{thebibliography}{HLT21}

\bibitem[Bad01]{badescu2001algebraic}
L.~Badescu, \emph{Algebraic surfaces}, Graduate Texts in Mathematics, Springer, 2001.

\bibitem[BB17]{MR3592689}
Arend Bayer and Tom Bridgeland, \emph{Derived automorphism groups of {K}3 surfaces of {P}icard rank 1}, Duke Math. J. \textbf{166} (2017), no.~1, 75--124. \MR{3592689}

\bibitem[BM77]{MR0491719}
E.~Bombieri and D.~Mumford, \emph{Enriques' classification of surfaces in char. {$p$}. {II}}, Complex analysis and algebraic geometry, Iwanami Shoten Publishers, Tokyo, 1977, pp.~23--42. \MR{491719}

\bibitem[BM01]{MR1827500}
Tom Bridgeland and Antony Maciocia, \emph{Complex surfaces with equivalent derived categories}, Math. Z. \textbf{236} (2001), no.~4, 677--697. \MR{1827500}

\bibitem[BM17]{MR3713877}
\bysame, \emph{Fourier-{M}ukai transforms for quotient varieties}, J. Geom. Phys. \textbf{122} (2017), 119--127. \MR{3713877}

\bibitem[BO01]{MR1818984}
Alexei Bondal and Dmitri Orlov, \emph{Reconstruction of a variety from the derived category and groups of autoequivalences}, Compositio Math. \textbf{125} (2001), no.~3, 327--344. \MR{1818984}

\bibitem[Bri98]{MR1629929}
Tom Bridgeland, \emph{Fourier-{M}ukai transforms for elliptic surfaces}, J. Reine Angew. Math. \textbf{498} (1998), 115--133. \MR{1629929}

\bibitem[GM13]{gelfand2013methods}
S.~I. Gelfand and Y.~I. Manin, \emph{Methods of homological algebra}, Springer Science \& Business Media, 2013.

\bibitem[HL10]{huybrechts2010geometry}
Daniel Huybrechts and Manfred Lehn, \emph{The geometry of moduli spaces of sheaves}, Cambridge University Press, 2010.

\bibitem[HLT17]{HonigsLieblichTirabassi2017arXiv1708.03409v1}
Katrina Honigs, Max Lieblich, and Sofia Tirabassi, \emph{Fourier--mukai partners of enriques and bielliptic surfaces in positive characteristic}, 2017, arXiv:1708.03409v1 [math.AG] (version submitted on 10 Aug 2017).

\bibitem[HLT21]{MR4247995}
\bysame, \emph{Fourier-{M}ukai partners of {E}nriques and bielliptic surfaces in positive characteristic}, Math. Res. Lett. \textbf{28} (2021), no.~1, 65--91. \MR{4247995}

\bibitem[Huy06]{MR2244106}
D.~Huybrechts, \emph{Fourier-{M}ukai transforms in algebraic geometry}, Oxford Mathematical Monographs, The Clarendon Press, Oxford University Press, Oxford, 2006. \MR{2244106}

\bibitem[IU05]{MR2198807}
Akira Ishii and Hokuto Uehara, \emph{Autoequivalences of derived categories on the minimal resolutions of {$A_n$}-singularities on surfaces}, J. Differential Geom. \textbf{71} (2005), no.~3, 385--435. \MR{2198807}

\bibitem[Kaw02]{10.4310/jdg/1090351323}
Yujiro Kawamata, \emph{{D-Equivalence and K-Equivalence}}, Journal of Differential Geometry \textbf{61} (2002), no.~1, 147 -- 171.

\bibitem[KO95]{S_A_Kuleshov_1995}
S.~A. Kuleshov and D.~O. Orlov, \emph{Exceptional sheaves on del pezzo surfaces}, Izvestiya: Mathematics \textbf{44} (1995), no.~3, 479.

\bibitem[Kos22]{10.1093/imrn/rnac260}
Naoki Koseki, \emph{On the {B}ogomolov–{G}ieseker inequality in positive characteristic}, International Mathematics Research Notices \textbf{2023} (2022), no.~24, 20784--20811.

\bibitem[Muk81]{nmj/1118786312}
Shigeru Mukai, \emph{{Duality between $D(X)$ and $D(\hat X)$ with its application to Picard sheaves}}, Nagoya Mathematical Journal \textbf{81} (1981), 153 -- 175.

\bibitem[Orl02]{MR1921811}
D.~O. Orlov, \emph{Derived categories of coherent sheaves on abelian varieties and equivalences between them}, Izv. Ross. Akad. Nauk Ser. Mat. \textbf{66} (2002), no.~3, 131--158. \MR{1921811}

\bibitem[Orl03]{MR1998775}
\bysame, \emph{Derived categories of coherent sheaves and equivalences between them}, Uspekhi Mat. Nauk \textbf{58} (2003), no.~3(351), 89--172. \MR{1998775}

\bibitem[Pin05]{Pink-FiniteGroupSchemes}
Richard Pink, \emph{Finite commutative group schemes}, Lecture notes, 2005, Lecture course at ETH Zürich, 2004--2005.

\bibitem[Pot15]{potter2017derived}
Rory Potter, \emph{Derived autoequivalences of bielliptic surfaces}, 2017, arXiv:1701.01015.

\bibitem[Ser90]{MR1038716}
Fernando Serrano, \emph{Divisors of bielliptic surfaces and embeddings in {${\bf P}^4$}}, Math. Z. \textbf{203} (1990), no.~3, 527--533. \MR{1038716}

\bibitem[Ueh16]{MR3568337}
Hokuto Uehara, \emph{Autoequivalences of derived categories of elliptic surfaces with non-zero {K}odaira dimension}, Algebr. Geom. \textbf{3} (2016), no.~5, 543--577. \MR{3568337}

\bibitem[Ueh17]{MR3652778}
\bysame, \emph{Fourier-{M}ukai partners of elliptic ruled surfaces}, Proc. Amer. Math. Soc. \textbf{145} (2017), no.~8, 3221--3232. \MR{3652778}

\bibitem[UW24]{uehara2024fouriermukaipartnersellipticruled}
Hokuto Uehara and Tomonobu Watanabe, \emph{Fourier--{M}ukai partners of elliptic ruled surfaces over arbitrary characteristic fields}, 2024.

\end{thebibliography}
